\documentclass{article}

\usepackage{amsmath,amssymb,amsthm,bm}
\usepackage{bbm}
\usepackage{enumerate}
\usepackage{pifont}
\usepackage{setspace}
\usepackage{tikz-cd} 
\usepackage[utf8]{inputenc}

\newcommand{\ZB}{\mathbb{Z}}
\newcommand{\DB}{\mathbb{D}}

\newcommand{\RB}{\mathbb{R}} 
\newcommand{\TB}{\mathbb{T}} 
\newcommand{\CB}{\mathbb{C}}
\newcommand{\1}{\mathbbm{1}}

\newcommand{\DC}{\mathcal{D}}
\newcommand{\EC}{\mathcal{E}}

\newcommand{\HC}{\mathcal{H}}

\newcommand{\LC}{\mathcal{L}}

\newcommand{\PC}{\mathcal{P}}

\newcommand*\conj[1]{\overline{#1}}
\newcommand*\clos[1]{\overline{#1}}
\newcommand*\LHS{\mathrm{LHS}}
\newcommand*\RHS{\mathrm{RHS}}

\renewcommand{\Re}{\operatorname{Re}}

\makeatletter
\providecommand*{\diff}%
{\@ifnextchar^{\DIfF}{\DIfF^{}}}
\def\DIfF^#1{%
	\mathop{\mathrm{\mathstrut d}}%
	\nolimits^{#1}\gobblespace}
\def\gobblespace{%
	\futurelet\diffarg\opspace}
\def\opspace{%
	\let\DiffSpace\!
	\ifx\diffarg(%
	\let\DiffSpace\relax
	\else
	\ifx\diffarg[%
	\let\DiffSpace\relax
	\else
	\ifx\diffarg\{%
	\let\DiffSpace\relax
	\fi\fi\fi\DiffSpace}
\makeatother

\theoremstyle{plain}
\newtheorem{thm}{Theorem}[section]
\newtheorem{lem}[thm]{Lemma}
\newtheorem{prop}[thm]{Proposition}
\newtheorem{cor}[thm]{Corollary}

\theoremstyle{definition}

\newtheorem{ex}[thm]{Example}

\theoremstyle{remark}
\newtheorem{rem}[thm]{Remark}

\numberwithin{equation}{section}

\title{On $m$-isometric semigroups, and $2$-isometric cogenerators}
\author{Eskil Rydhe\thanks{e.rydhe@leeds.ac.uk. School of Mathematics, University of Leeds, Leeds LS2 9JT, UK. This work was supported by the Knut and Alice Wallenberg foundation, scholarship KAW 2016.0442.}}

\begin{document}

\maketitle

\begin{abstract}
	It is known that a $C_0$-semigroup of Hilbert space operators is $m$-isometric if and only if its generator satisfies a certain condition, which we choose to call $m$-skew-symmetry. This paper contains two main results: We provide a Lumer--Phillips type characterization of generators of $m$-isometric semigroups. This is based on the simple observation that $m$-isometric semigroups are quasicontractive. We also characterize cogenerators of $2$-isometric semigroups. To this end, our main strategy is to construct a functional model for $2$-isometric semigroups with analytic cogenerators. The functional model yields numerous simple examples of non-unitary $2$-isometric semigroups, but also allows for the construction of a closed, densely defined, $2$-skew-symmetric operator which is not a semigroup generator.
\end{abstract}

\maketitle
\section{Introduction}

Throughout this paper, we let $\HC$ and $\EC$ denote complex Hilbert spaces. On their own, these are assumed to be completely general, while together, we will typically let $\EC$ denote a certain subspace of $\HC$. The word \textit{operator} always refers to a linear map $A:D(A)\to \HC$, with domain $D(A)\subseteq\HC$. While all operators encountered below are closed, we do not insist on this as part of the definition. The space of bounded operators is denoted $\LC=\LC(\HC)$. As for operators on $\EC$, we are more interested in $\LC_+=\LC_+(\EC)$, the convex cone of positive bounded operators.

A bounded operator $T$ is called \textit{$m$-isometric} if
\begin{align*}
\sum_{k=0}^m (-1)^{m-k}\binom{m}{k}\|T^kx\|^2=0,\quad x\in\HC.
\end{align*}
The study of $m$-isometries was initiated by Agler~\cite{Agler1990:ADisconjugacyTheoremForToeplitzOperators}, and significantly extended by Agler and Stankus~\cite{Agler-Stankus1995:m-IsometricTransformationsOfHilbertSpaceI,Agler-Stankus1995:m-IsometricTransformationsOfHilbertSpaceII,Agler-Stankus1995:m-IsometricTransformationsOfHilbertSpaceIII}. A related notion is that of $m$-symmetric operators, introduced by Helton~\cite{Helton1972:OperatorsWithARepresentationAsMultiplicationByXOnASobolevSpace}. By analogy, we say that an operator $A$ is \textit{$m$-skew-symmetric} if
\begin{align*}
\sum_{k=0}^m\binom{m}{k}\left\langle A^ky,A^{m-k}y\right\rangle = 0,\quad y\in D(A^m).
\end{align*}
This happens precisely when $iA$ is $m$-symmetric.

If we think of $n\in\ZB_{\ge 0}$ as a time variable, then the sequence $(T^n)_{n=0}^\infty$ constitutes a semigroup describing the evolution of some abstract linear system in discrete time. It can be shown that if $T$ is an $m$-isometry, then the semigroup $(T^n)_{n=0}^\infty$ consists entirely of $m$-isometries. A recent development is the study of $m$-isometric semigroups in continuous time, i.e. $C_0$-semigroups of $m$-isometries~\cite{Bermudez-Bonilla-Zaway2019:C0-SemigroupsOfM-IsometriesOnHilbertSpaces,GallardoGutierrez-Partington2018:C0-SemigroupsOf2-IsometriesAndDirichletSpaces}. It is known that a $C_0$-semigroup is $m$-isometric if and only if its generator is $m$-skew-symmetric. Moreover, this is equivalent to that, for each $x\in\HC$, the function $t\mapsto \|T_tx\|^2$ is a polynomial of degree no more than $(m-1)$ \cite[Theorem~2.1]{Bermudez-Bonilla-Zaway2019:C0-SemigroupsOfM-IsometriesOnHilbertSpaces}. We will prove a slightly more specific result about the function $t\mapsto \|T_tx\|^2$ (Theorem~\ref{thm:PolynomialGrowthContinuousTime}). This readily implies that $m$-isometric semigroups are quasicontractive (Corollary~\ref{cor:QuasiContractive}). Our first main result is then a simple consequence of the Lumer--Phillips theorem:
\begin{thm}\label{thm:MainGen}
	Let $m\in\ZB_{\ge 1}$. An operator $A$ is the generator of an $m$-isometric semigroup $(T_t)_{t\ge 0}$ if and only if the following conditions are satisfied:
	\begin{enumerate}[(i)]
		\item $A$ is closed and densely defined.
		\item There exists $w\ge 0$ such that
		\begin{align}\label{eq:QCCondGen}
		\Re\left\langle Ay,y\right\rangle \le w\|y\|^2,\quad y\in D(A).
		\end{align}
		\item There exists $\lambda>w$ such that $\lambda -A:D(A)\to\HC$ is surjective.
		\item $A$ is $m$-skew-symmetric.
	\end{enumerate}
	If the above conditions hold, and if $w\ge 0$, then \eqref{eq:QCCondGen} is satisfied if and only if $(T_t)_{t\ge 0}$ is quasicontractive with parameter $w$. Moreover, if $\lambda>0$, then $\lambda-A:D(A)\to\HC$ is invertible.
\end{thm}

The final assertion of the above theorem follows from the polynomial growth of $\|T_tx\|^2$, together with the standard resolvent formula \eqref{eq:LaplaceTransformOfSemigroup}. Consequently,  the spectrum of $A$ lies in the left half of the complex plane. Therefore, $(T_t)_{t\ge 0}$ has a well-defined \textit{cogenerator} $T\in\LC$ given by
\begin{align*}
T=(A+I)(A-I)^{-1}.
\end{align*}
The cogenerator satisfies certain necessary conditions:
\begin{thm}\label{thm:MainCoGen}
	Let $m\in\ZB_{\ge 1}$. If $(T_t)_{t\ge 0}$ is an $m$-isometric semigroup, then $(T_t)_{t\ge 0}$ has a well-defined cogenerator $T\in\LC$ satisfying the following conditions:
	\begin{enumerate}[(i)]
		\item $T-I$ is injective, and has dense range.
		\item There exists $w\ge 0$ such that
		\begin{align}\label{eq:QCCondCoGen}
		\|Tx\|^2-\|x\|^2 \le w\|(T-I)x\|^2,\quad x\in \HC.
		\end{align}	
		\item $T$ is $m$-isometric.
	\end{enumerate}
	Moreover, \eqref{eq:QCCondCoGen} holds whenever $(T_t)_{t\ge 0}$ is quasicontractive with parameter $w$.
\end{thm}

It is classical, e.g. \cite[Chapter~III, Section~8]{Sz.-Nagy-Foias-Bercovici-Kerchy2010:HarmonicAnalysisOfOperatorsOnHilbertSpace}, that if $T\in\LC$ is a contraction, and $T-I$ is injective, then $T$ is the cogenerator of a contractive $C_0$-semigroup. Moreover, the semigroup is isometric whenever $T$ is. Note that if $T$ is isometric, then \eqref{eq:QCCondCoGen} is trivially satisfied for any $w\ge 0$. Hence, Theorem~\ref{thm:MainCoGen} has a natural converse in the case where $m=1$. We obtain a similar result for $m=2$:

\begin{thm}\label{thm:MainConverse}
	Let $T\in\LC$ be a $2$-isometry. Assume further that $1\notin\sigma_p(T)$, and that there exists $w\ge 0$ for which \eqref{eq:QCCondCoGen} is satisfied. Then $T$ is the cogenerator of a $C_0$-semigroup $(T_t)_{t\ge 0}$, which is $2$-isometric, and quasicontractive with parameter $w$.
\end{thm}

Note that Theorem~\ref{thm:MainConverse} does not require $T-I$ to have dense range. This is much like in the isometric case: If $T$ is an isometry, then $T$ and $T^*$ have the same invariant vectors, so $T-I$ has dense range if and only if it is injective. Our results imply that $T-I$ has dense range under the more general hypothesis of Theorem~\ref{thm:MainConverse}, but the author has not found a direct proof 	.

The proof of Theorem~\ref{thm:MainConverse} relies primarily on two non-trivial results about $2$-isometries, namely a Wold-type decomposition theorem, and a functional model for analytic $2$-isometries.

The first of these results (Theorem~\ref{thm:WoldDecomposition}) is due to Shimorin \cite{Shimorin2001:Wold-TypeDecompositionsAndWanderingSubspacesForOperatorsCloseToIsometries}, and states that any $2$-isometry $T$ can be written as a direct sum of a unitary operator, and an analytic $2$-isometry. By some classical properties of unitary cogenerators, this allows us to reduce the proof of Theorem \ref{thm:MainConverse} to the case where $T$ is analytic.

The second result (Theorem~\ref{thm:RichterOlofssonFunctionalModel}) states that if $T\in\LC=\LC(\HC)$ is an analytic $2$-isometry, and $\EC=\HC\ominus T\HC$, then $T$ is unitarily equivalent to the operator $M_z$, multiplication by the identity function $z\mapsto z$, acting on a harmonically weighted Dirichlet space $\DC_\mu ^2(\EC)$ of $\EC$-valued analytic functions on $\DB$. The parameter $\mu$ is a measure on the unit circle $\TB$, with values in $\LC_+=\LC_+(\EC)$. The correspondence between $T$ and $\mu$ is essentially bijective. Theorem~\ref{thm:RichterOlofssonFunctionalModel} was proved by Richter \cite{Richter1991:ARepresentationTheoremForCyclicAnalyticTwo-Isometries} in the case where $\dim \EC=1$, and extended to the general case by Olofsson \cite{Olofsson2004:AVonNeumann-WoldDecompositionOfTwo-Isometries}. The main idea behind Theorem~\ref{thm:MainConverse} is that an analytic $2$-isometry $T\in\LC$ is the cogenerator of a $C_0$-semigroup if and only if $M_z$ is the cogenerator of a $C_0$-semigroup on the corresponding space $\DC_{\mu}^2(\EC)$. Such a semigroup is necessarily given by the multiplication operators $(M_{\phi_t})_{t\ge 0}$, where $\phi_t:z\mapsto\exp\left(t(z+1)/(z-1)\right)$. Our strategy is then to determine all $\mu$ such that the operators $(M_{\phi_t})_{t\ge 0}$ form a $C_0$-semigroup on $\DC_{\mu}^2(\EC)$. 

After reduction to the analytic case, Theorem \ref{thm:MainConverse} follows from:
\begin{thm}\label{thm:MainMeasures}
	Let $\mu$ be an $\LC_+$-valued measure on $\TB$. Then the following are equivalent:
	\begin{enumerate}[(i)]
		\item For every $t\ge 0$, $M_{\phi_t}\in\LC(\DC_\mu^2(\EC))$, and the family $(M_{\phi_t})_{t\ge0}$ is a $C_0$-semigroup on $\DC_\mu^2(\EC)$.
		\item There exists $w_1\ge 0$ such that 
		\begin{align}\label{eq:Condw1}
		\frac{1}{2\pi}\int_{\TB}\left\langle \diff\mu\, f,f\right\rangle
		\le
		w_1
		\|(I-M_z)f\|_{\DC_\mu^2(\EC)}^2,\quad f\in\PC_a(\EC).
		\end{align}
		\item The set function $\tilde \mu: E\mapsto \frac{1}{2\pi}\int_{E}\frac{\diff\mu(\zeta)}{|1-\zeta|^2}$ is an $\LC_+$-valued measure, and there exists $w_2\ge 0$ such that 
		\begin{align}\label{eq:Condw2}
		\frac{1}{2\pi}\int_{\TB}\left\langle \diff\tilde{\mu}\, f,f\right\rangle
		\le
		w_2
		\|f\|_{\DC_\mu^2(\EC)}^2,\quad f\in\PC_a(\EC).
		\end{align}
	\end{enumerate}
	If either of the above conditions is satisfied, then the semigroup $(M_{\phi_t})_{t\ge0}$ is $2$-isometric, has cogenerator $M_z$, and is quasicontractive with some parameter $w\ge 0$. The optimal (smallest possible) values for $w$, $w_1$ and $w_2$ coincide.
\end{thm}

\begin{rem}
	For an $\LC_+$-valued measure $\mu$, and $x\in\EC$, the set function $\mu_{x,x}:E\mapsto\left\langle \mu(E)x,x\right\rangle$ defines a finite positive Borel measure. If the inequality \eqref{eq:Condw2} holds for constant functions, then
	\begin{equation}\label{eq:Condw2'}
	\frac{1}{2\pi}\int_{\TB} \frac{\diff\mu_{x,x}(\zeta)}{|1-\zeta|^2}
	\le
	w_2
	\|x\|^2,\quad x\in\EC.
	\end{equation}
	On the other hand, if $\mu$ is any measure satisfying the above inequality, then $E\mapsto \int_{E}\frac{\diff\mu(\zeta)}{|1-\zeta|^2}$ defines an $\LC_+$-valued measure. Hence, when attempting to verify condition $(iii)$ of Theorem~\ref{thm:MainMeasures}, the inequality \eqref{eq:Condw2'} is a natural first step.
\end{rem}

Any bounded operator $A$ generates an invertible $C_0$-semigroup $(e^{tA})_{t\ge 0}$. As a special case of \cite[Theorem~2.1]{Bermudez-Bonilla-Zaway2019:C0-SemigroupsOfM-IsometriesOnHilbertSpaces} (or Theorem~\ref{thm:MainGen}), stated explicitly as \cite[Corollary~2.3]{Bermudez-Bonilla-Zaway2019:C0-SemigroupsOfM-IsometriesOnHilbertSpaces}, the semigroup is $m$-isometric if and only if $A$ is $m$-skew-symmetric. A limitation of this conclusion is that if $m$ is even, then by \cite[Proposition~1.23]{Agler-Stankus1995:m-IsometricTransformationsOfHilbertSpaceI}, $(e^{tA})_{t\ge 0}$ is in fact $(m-1)$-isometric. In particular, any $2$-isometric $C_0$-semigroup with bounded generator is unitary. On the other hand, Theorem~\ref{thm:MainMeasures} allows one to produce numerous examples of non-unitary $2$-isometric semigroups. In particular, such semigroups exist. One may also use Theorem~\ref{thm:MainMeasures} to construct a closed, densely defined, $2$-skew-symmetric operator $A$, with the property that $\lambda-A:D(A)\to\HC$ is surjective for any $\lambda>0$, but which is not the generator of a $C_0$-semigroup. This shows that the conditions \eqref{eq:QCCondGen} and \eqref{eq:QCCondCoGen} are not superfluous.

The paper is organized as follows: In Section~\ref{sec:NotationAndPrelimiaries} we introduce some notation and preliminary material. In Section~\ref{sec:m-IsometricSemigroups} we discuss $m$-isometric semigroups. In particular, we prove Theorem~\ref{thm:MainGen}, and Theorem~\ref{thm:MainCoGen}. In Section~\ref{sec:2-isometricCogenerators} we prove Theorem~\ref{thm:MainConverse}, and Theorem~\ref{thm:MainMeasures}. In Section~\ref{sec:Examples} we discuss some examples related to Theorem~\ref{thm:MainMeasures}. In Section~\ref{sec:m-ConcaveSemigroups} we briefly mention the wider context of $m$-concave semigroups.

\section{Notation and preliminaries}\label{sec:NotationAndPrelimiaries}

We use the notation $\DB:=\{z\in\CB;\, |z|<1\}$ for the open unit disc,  and $\TB:=\{z\in\CB;\,|z|=1\}$ for the unit circle of the complex plane $\CB$. By $\lambda$ we denote Lebesgue (arc length) measure on $\TB$, while $\diff A$ will signify integration with respect to area measure on $\CB$. We also use $\partial\Omega$ and $\clos{\Omega}$ to denote the boundary and closure of $\Omega\subset\CB$, respectively.

Given integers $m\ge k\ge 0$, we let $\binom{m}{k}=\frac{m!}{k!(m-k)!}$ denote the standard binomial coefficients. If the integers $m$ and $k$ do not satisfy the prescribed inequalities, then we set $\binom{m}{k}=0$. With this convention, the well-known relation
\begin{align}\label{eq:BinomialIdentity}
\binom{m}{k}=\binom{m-1}{k-1}+\binom{m-1}{k}
\end{align}
is valid whenever $m\ge 1$ and $k\in\ZB$. This will be used repeatedly.

Given an operator $A$, we let $\sigma(A)$, $\sigma_p(A)$, $\sigma_{ap}(A)$, and $W(A)$ respectively denote the spectrum, point spectrum, approximate point spectrum, and numerical range of $A$, i.e.
\begin{align*}
\sigma(A)&:=\{z\in\CB;\,z-A:D(A)\to\HC\textnormal{ is not bijective}\},\\
\sigma_p(A)&:=\{z\in\CB;\,z-A:D(A)\to\HC\textnormal{ is not injective}\},\\
\sigma_{ap}(A)&:=\{z\in\CB;\,z-A:D(A)\to\HC\textnormal{ is not bounded below}\},\\
W(A)&:=\left\{\left\langle Ay,y\right\rangle_\HC \in\CB ;\,y\in D(A),\|y\|=1\right\}.
\end{align*}
We also let $\rho(A):=\CB\setminus\sigma(A)$. If $z\in\rho(A)$, and $A$ is closed, then $(z-A)^{-1}:\HC\to D(A)$ is bounded by the closed graph theorem. On the other hand, if $(z-A)^{-1}$ is bounded, then $A$ is closed.

Given a family $\left\{S_i\right\}_{i\in I}$ of subsets of $\HC$, we let $\bigvee_{i\in I}S_i$ denote the smallest closed subspace of $\HC$ that contains each $S_i$.

\subsection{$m$-isometries, and $m$-skew-symmetries}
Let $m\in\ZB_{\ge 0}$, $T\in\LC$, and define 
\begin{align*}
\beta_{m}(T)=\sum_{j=0}^{m}(-1)^{m-j}\binom{m}{j}T^{*j}T^j.
\end{align*}
Apart from a normalizing factor, this agrees with the notation from \cite{Agler-Stankus1995:m-IsometricTransformationsOfHilbertSpaceI}. A straightforward consequence of \eqref{eq:BinomialIdentity} is that 
\begin{align}\label{eq:BetaRecursion}
\beta_{m+1}(T)=T^*\beta_m(T)T-\beta_m(T).
\end{align}
We also have the following formula, valid for $k\in\ZB_{\ge 0}$, $T\in\LC$:
\begin{align}\label{eq:SumOfDefectOperators}
T^{*k}T^k=\sum_{j=0}^{\infty}\binom{k}{j}\beta_{j}(T).
\end{align}
Note that with our notational convention for binomial coefficients, the above right-hand side has at most $k+1$ non-zero terms.

If $\beta_{m}(T)=0$, then we say that $T$ is an \textit{$m$-isometry}. By \eqref{eq:BetaRecursion}, any such operator is also an $(m+1)$-isometry. Moreover, \eqref{eq:SumOfDefectOperators} implies
\begin{align}\label{eq:PolynomialGrowthDiscreteTime}
\|T^kx\|^2=\sum_{j=0}^{m-1}\binom{k}{j}\left\langle \beta_{j}(T)x,x\right\rangle ,\quad x\in\HC.
\end{align}
Thus, $\|T^kx\|^2$ is polynomial in $k$ whenever $x\in\HC$, and $\|T^k\|^2\lesssim (1+k)^{m-1}$. Gelfand's formula for the spectral radius yields that $\sigma (T)\subseteq \clos{\DB}$. A more careful analysis reveals that $\sigma_{ap}(T)\subseteq \TB$, see \cite[Lemma~1.21]{Agler-Stankus1995:m-IsometricTransformationsOfHilbertSpaceI}. By the general fact that $\partial\sigma(T)\subseteq\sigma_{ap}(T)$, it follows that an $m$-isometry is either invertible, in which case $\sigma(T)\subseteq \TB$, or it is not invertible, in which case $\sigma(T)=\clos{\DB}$.

Given $m\in\ZB_{\ge 1}$, and an operator $A$, we define the sesquilinear form $\alpha_m^A$ by
\begin{align}\label{eq:DefinitionOfAlphaM}
\alpha_m^A(y_1,y_2)=\sum_{j=0}^m\binom{m}{j}\left\langle A^jy_1,A^{m-j}y_2\right\rangle,\quad y_1,y_2\in D(A^m).
\end{align}
It will be convenient to write $\alpha_{m}^A(y)$ instead of $\alpha_{m}^A(y,y)$. By \eqref{eq:BinomialIdentity}, 
\begin{align}\label{eq:AlphaRecursion}
\alpha_{m+1}^A(y_1,y_2)=\alpha_m^A(Ay_1,y_2)+\alpha_m^A(y_1,Ay_2).
\end{align}
We say that $A$ is \textit{$m$-skew-symmetric} if $\alpha_m^A$ vanishes identically. 

The relation between $m$-isometries and $m$-skew-symmetries has been frequently exploited in previous works \cite{Bermudez-Bonilla-Zaway2019:C0-SemigroupsOfM-IsometriesOnHilbertSpaces,Jacob-Partington-Pott-Wynn2015:Beta-AdmissibilityForGamma-HypercontractiveSemigroups}. For easy reference, we state and prove the following result, which is implicit in \cite[p.~425]{Rydhe2016:AnAgler-TypeModelTheoremForC0-SemigroupsOfHilbertSpaceContractions}:
\begin{lem}\label{lemma:M-IsometryToM-SkewSymmetry}
	Let $A$ be a closed operator, with $1\in\rho(A)$, and define $T\in\LC$ by
	\[
	T=(A+I)(A-I)^{-1}.
	\]
	Then $T-I$ is injective, $D(A)=(T-I)\HC$, and $A=(T+I)(T-I)^{-1}$. Moreover, for $m\in\ZB_{\ge 0}$, and $x_1,x_2\in\HC$,
	\begin{align}\label{eq:BetaAndAlphaFormula}
	\left\langle \beta_m(T) x_1,x_2\right\rangle =2^m\alpha_m^A\left((A-I)^{-m}x_1,(A-I)^{-m}x_2\right).
	\end{align}
\end{lem}
\begin{proof}
	Since $A$ is closed, $(A-I)^{-1}:\HC\to D(A)$ is bounded. It is easy to see that $T=I+2(A-I)^{-1}$. Hence, the bounded operator $T-I$ is injective, with left-inverse $\frac{1}{2}(A-I)$. Moreover, $D(A)=(T-I)\HC$. The fact that $A=(T+I)(T-I)^{-1}$ is a simple algebraic verification.
	
	It is clear that \eqref{eq:BetaAndAlphaFormula} holds for $m=0$. For a general $m$, we use \eqref{eq:BetaRecursion} to see that
	\begin{multline*}
	\left\langle \beta_m(T)(T+I)x_1,(T-I)x_2\right\rangle 
	\\
	=
	\left\langle \beta_{m+1}(T)x_1,x_2\right\rangle
	-
	\left\langle \beta_{m}(T)Tx_1,x_2\right\rangle
	+
	\left\langle \beta_{m}(T)x_1,Tx_2\right\rangle
	\end{multline*}
	Since $\beta_m(T)^*=\beta_m(T)$, the above identity, with $x_1$ and $x_2$ interchanged, implies
	\begin{multline*}
	\left\langle \beta_m(T)(T-I)x_1,(T+I)x_2\right\rangle 
	\\
	=
	\left\langle \beta_{m+1}(T)x_1,x_2\right\rangle
	+
	\left\langle \beta_{m}(T)Tx_1,x_2\right\rangle
	-
	\left\langle \beta_{m}(T)x_1,Tx_2\right\rangle
	\end{multline*}
	Adding these two identities,
	\begin{multline*}
	2\left\langle \beta_{m+1}(T)x_1,x_2\right\rangle
	\\
	=
	\left\langle \beta_m(T)(T+I)x_1,(T-I)x_2\right\rangle
	+
	\left\langle \beta_m(T)(T-I)x_1,(T+I)x_2\right\rangle.
	\end{multline*}
	Since $T+I=2A(A-I)^{-1}$, and $T-I=2(A-I)^{-1}$, we therefore have
	\begin{multline*}
	\left\langle \beta_{m+1}(T)x_1,x_2\right\rangle
	=
	2\left\langle \beta_m(T)A(A-I)^{-1}x_1,(A-I)^{-1}x_2\right\rangle
	\\
	+
	2\left\langle \beta_m(T)(A-I)^{-1}x_1,A(A-I)^{-1}x_2\right\rangle.
	\end{multline*}
	Assuming that \eqref{eq:BetaAndAlphaFormula} holds, \eqref{eq:AlphaRecursion} implies 
	\[
	\left\langle \beta_{m+1}(T) x_1,x_2\right\rangle =2^{m+1}\alpha_{m+1}^A\left((A-I)^{-m-1}x_1,(A-I)^{-m-1}x_2\right).
	\]
	Hence, we obtain \eqref{eq:BetaAndAlphaFormula} by induction over $m$.
\end{proof}

\subsection{$C_0$-semigroups}\label{subsec:C_0-Semigroups}

By a \textit{semigroup} we mean a one-parameter family $(T_t)_{t\ge 0}\subset \LC$, such that $T_0=I$, and $T_{s+t}=T_sT_t$ for $s,t\ge 0$. For a detailed treatment of the facts outlined below, we refer to \cite[Chapter II]{Engel-Nagel2000:One-ParameterSemigroupsForLinearEvolutionEquations}.

A semigroup is called a \textit{$C_0$-semigroup}, or \textit{strongly continuous}, if for every $x\in\HC$ the \textit{orbit map} $\xi_x:[0,\infty)\ni t\mapsto T_tx\in\HC$ is continuous. Given a $C_0$-semigroup $(T_t)_{t\ge 0}$, the uniform boundedness principle implies that $(T_t)_{0\le t\le 1}$ is a bounded family in $\LC$. The semigroup property then implies that $(T_t)_{t\ge 0}$ is exponentially bounded, i.e. there exists $M\ge 1$ and $w\in\RB$ such that 
\begin{align}\label{eq:ExponentialBoundedness}
\|T_t\|_\LC\le M e^{wt},\quad t\ge 0.
\end{align}
If \eqref{eq:ExponentialBoundedness} holds with $M=1$, then we say that $(T_t)_{t\ge 0}$ is \textit{quasicontractive with parameter $w$}. A quasicontractive semigroup with parameter $0$ is simply called \textit{contractive}.

The \textit{(infinitesimal) generator} of $(T_t)_{t\ge 0}$ is the operator $A$ defined by
\[
Ay=\lim_{t\to 0^+}\frac{T_ty-y}{t}.
\]
Its domain $D(A)$ is the subspace of $y\in\HC$ such that the above limit exists. The generator of a $C_0$-semigroup is closed, densely defined, and uniquely determines the semigroup. If $y\in D(A)$ and $t\ge 0$, then $T_ty\in D(A)$, and $AT_ty=T_tAy$.

Let $x\in\HC$, and $\lambda >w$, where $w$ is as in \eqref{eq:ExponentialBoundedness}. Since $t\mapsto T_tx$ is continuous, the integral $\int_0^\infty T_txe^{-\lambda t}\diff t$ is well-defined as a generalized Riemann integral. It is easy to show that in this sense,
\begin{align}\label{eq:LaplaceTransformOfSemigroup}
(\lambda - A)^{-1}=\int_0^\infty T_te^{-\lambda t}\diff t.
\end{align}

Not every closed, densely defined operator is the generator of a $C_0$-semigroups. A fundamental result in this direction is the so-called Lumer--Phillips theorem:
\begin{thm}\label{thm:Lumer--Phillips}
	Let $w\ge 0$, and $A$ be an operator. Then $A$ is the generator of a quasicontractive $C_0$-semigroup with parameter $w$ if and only if the following conditions hold:
	\begin{enumerate}[(i)]
		\item $A$ is closed and densely defined.
		\item $A$ is $w$-dissipative, i.e. 
		\begin{align*}
		\Re \left\langle Ay,y\right\rangle\le w\|y\|^2,\quad y\in D(A).
		\end{align*}
		\item There exists $\lambda>w$ such that $\lambda-A:D(A)\to\HC$ is surjective.	
	\end{enumerate}
\end{thm}
\begin{rem}
	The above theorem is typically stated for $w=0$. The general version follows if we consider the operator $A-w$, and the corresponding semigroup $(e^{-wt}T_t)_{t\ge 0}$.
	
	An elementary calculation using the inner product shows that $A$ is $0$-dissipative if and only if
	\begin{align*}
	\|(\lambda -A)y\|\ge \lambda \|y\|,\quad y\in D(A),\lambda >0.
	\end{align*}
	Consequently, $\lambda -A$ is injective in this case. Moreover, the above inequality turns out to be the appropriate analogue of $(ii)$ when studying semigroups of operators on Banach spaces.
	
	If the above conditions $(i)-(iii)$ hold, then $\lambda-A:D(A)\to\HC$ is in fact surjective for \textit{any} $\lambda>w$.
\end{rem}

\subsection{Operator measures}\label{subsec:OperatorMeasures}

Let $\mathfrak{S}$ denote the Borel $\sigma$-algebra of subsets of $\TB$. An \textit{$\LC_+$-valued measure} is a finitely additive set function $\mu:\mathfrak{S}\to\LC_+$ with the property that for every $x,y\in\EC$, the set function $\mu_{x,y}:E\mapsto \left\langle \mu(E)x,y\right\rangle$ defines a complex regular Borel measure. For each $E\in\mathfrak{S}$, it holds that
\begin{align}\label{eq:CauchySchwarzForOperatorMeasures}
|\mu_{x,y}|(E)\le \mu_{x,x}(E)^{1/2}\mu_{y,y}(E)^{1/2}.
\end{align}
We refer to the proof of \cite[Proposition 1.1]{Olofsson2004:OperatorValuedN-HarmonicMeasureInThePolydisc}.

Given a bounded (Borel) measurable function $f:\TB\to\CB$, we can define the sesquilinear form $J_f:(x,y)\mapsto \int_{\TB}f\diff\mu_{x,y}$. It follows from \eqref{eq:CauchySchwarzForOperatorMeasures}, that
\begin{align*}
|J_f(x,y)|\le \|f\|_\infty\|\mu(\TB)\|\|x\|\|y\|.
\end{align*}
By standard functional analytic considerations, the above inequality implies the existence of a uniquely determined operator $I_f\in\LC$ such that $\left\langle I_f\, x,y\right\rangle =\int_{\TB}f\diff\mu_{x,y}$. We denote the operator $I_f$ by $\int_{\TB}f\diff\mu$. The integral thus defined satisfies the triangle type inequality
\[
\|\int_{\TB}f\diff\mu\|\le \|f\|_\infty \|\mu(\TB)\|.
\]

We will need the following version of the Cauchy--Schwarz inequality: 
\begin{lem}\label{lemma:VersionOfCauchySchwarz}
	Let $\mu$ be an $\LC_+$-valued measure. If $x,y\in\EC$, and $f,g:\TB\to\CB$ are Borel measurable functions, then
	\begin{align}\label{eq:CauchySchwarz}
	\int_\TB |fg|\diff|\mu_{x,y}|\le\left(\int_\TB |f|^2\diff\mu_{x,x}\right)^{1/2}\left(\int_\TB |g|^2\diff\mu_{y,y}\right)^{1/2}.
	\end{align}
\end{lem}
\begin{proof}
	When $f$ and $g$ are simple functions, \eqref{eq:CauchySchwarz} follows from \eqref{eq:CauchySchwarzForOperatorMeasures}, and the Cauchy--Schwarz inequality for finite sums. General functions are approximated in the standard fashion.
\end{proof}

Two instances of the above integral will be particularly interesting to us, namely the Fourier coefficients
\begin{align*}
\hat \mu(n)=\frac{1}{2\pi }\int_{\TB}\conj{\zeta^n}\diff\mu(\zeta),\quad n\in\ZB,
\end{align*}
and the Poisson extension
\begin{align*}
P_\mu(z)=\frac{1}{2\pi }\int_{\TB}\frac{1-|z|^2}{|\zeta-z|^2}\diff\mu(\zeta),\quad z\in\DB.
\end{align*}
If we let $r=|z|$, then 
\[
\frac{1-|z|^2}{|\zeta-z|^2}=\sum_{n\in\ZB} r^{|n|}\left(\frac{z}{r}\right)^n \conj{\zeta^n}.
\]
By the Weierstrass test, this series converges uniformly in $\zeta$. Using term by term integration, we conclude that 
\[
P_\mu(z) = \sum_{n\in\ZB} \hat{\mu}(n)r^{|n|}\left(\frac{z}{r}\right)^n.
\]

With the above construction, the integral $\int f\diff\mu$ is only defined when $f$ is bounded. As a remedy for this, adequate for our purposes, we use the following construction: Let $\mu$ be an $\LC_+$-valued measure, and $h$ a scalar-valued function. If there exists $C>0$ such that
\[
\int_{\TB}|h|\diff\mu_{x,x}\le C\|x\|^2,\quad x\in\EC,
\]
then one can define a new set function $\mu_h$ by
\[
\left\langle \mu_h(E)x,y\right\rangle =\int_E h\diff\mu_{x,y}.
\]
By Lemma \ref{lemma:VersionOfCauchySchwarz}, the above right-hand side has modulus less than
\[
\left(\int_\TB |h|\diff\mu_{x,x}\right)^{1/2}\left(\int_\TB |h|\diff\mu_{y,y}\right)^{1/2}\le C\|x\|\|y\|.
\]
This estimate implies that $\mu_h$ is another $\LC_+$-valued measure. If $f$ is bounded, then we may take $\int f\diff \mu_h$ as a definition of $\int fh\diff{\mu}$.

\begin{rem}
	We will only use the above construction of $\int fh\diff{\mu}$ in the setting where $h$ is a fixed function. However, the following may be of independent interest: If $f_1h_1=f_2h_2$, and the measures $\mu_i=\mu_{h_i}$ are defined as above, then
	\begin{align*}
	\left\langle \int f_i\diff{\mu}_i\, x,y\right\rangle
	=
	\int f_i\diff{(\mu_i)_{x,y}}
	=
	\int f_ih_i\diff{\mu_{x,y}}.
	\end{align*}
	Hence,
	\begin{align*}
	\int f_1h_1\diff{\mu}
	=
	\int f_2h_2\diff{\mu}.
	\end{align*}
\end{rem}

Let $f,g:\TB\to\EC$ be continuous functions, and identify these with their respective Poisson extensions. For $0<r<1$ and $\zeta\in\TB$, 
\begin{align*}
\left\langle P_\mu(r\zeta)f(r\zeta),g(r\zeta)\right\rangle = \sum_{k,l,n\in\ZB} \left\langle \hat \mu(n)\hat f(k),\hat g(l)\right\rangle r^{|n|+|k|+|l|}\zeta^{n+k-l},
\end{align*}
and the right-hand side converges uniformly in $\zeta$. Integrating with respect to $\diff\lambda(\zeta)$ yields
\[
\int_{\TB}\left\langle P_\mu(r\zeta )f(r\zeta),g(r\zeta)\right\rangle\diff \lambda(\zeta) = 2\pi \sum_{k,l\in\ZB} r^{|k|+|l|+|k-l|}\left\langle \hat{\mu}(l-k) \hat f(k),\hat g(l)\right\rangle.
\]
This motivates us to define $\int \left\langle\diff\mu\,\cdot,\cdot\right\rangle$ by
\begin{align}\label{eq:DefinitionOfQuadraticIntegral}
\frac{1}{2\pi}\int_{\TB}\left\langle \diff \mu \, f,g\right\rangle = \sum_{k,l\in\ZB} \left\langle \hat{\mu}(l-k) \hat f(k),\hat g(l)\right\rangle,
\end{align}
provided that the above right-hand side is absolutely convergent. 

It seems clear that any reasonable definition of $\int \left\langle\diff\mu\,\cdot,\cdot\right\rangle$ should satisfy \eqref{eq:DefinitionOfQuadraticIntegral}. On the other hand, it is a bit awkward to require so much regularity for a function to be square integrable. The next example is a digression from the primary topic of this paper, but may still be of interest.
\begin{ex}
	For $k\in\ZB_{\ge 1}$, let $I_k$ denote the arc $\{e^{it};t\in[2^{-k},2^{1-k})\}\subset \TB$, and define the set function $\mu$ by
	\begin{align*}
	\left\langle \mu(E)x,y\right\rangle =\sum_{k=1}^\infty 2^k\lambda(E\cap I_k)\left\langle x,e_k\right\rangle \left\langle e_k,y\right\rangle,\quad x,y\in\EC,
	\end{align*}
	where $(e_k)_{k=1}^\infty$ is some orthonormal sequence in $\EC$. Then $\mu$ is an $\LC_+$-valued measure. For a simple function $f=\sum_{n}x_n\1_{E_n}$, it might seem reasonable to define $\int \left\langle \diff \mu\, f,f\right\rangle$ as the finite sum
	\[\tag{naive}
	\sum_n \left\langle \mu(E_n)x_n,x_n\right\rangle.
	\]
	However, if $f_N=\sum_{k=1}^Ne_k\1_{I_k}$, then the above sum is equal to $N$, even though $\|f_N\|_{L^\infty}=1$. This may help explain why we have chosen \eqref{eq:DefinitionOfQuadraticIntegral} as our definition of $\int \left\langle \diff \mu\,\cdot,\cdot\right\rangle$.
\end{ex}

\subsection{Function spaces}

The space of analytic polynomials $\sum_{k=0}^{N}a_kz^k$ with coefficients $a_k\in\EC$ is denoted by $\PC_a(\EC)$. As a notational convention, we write $\PC_a$ in place of $\PC_a(\CB)$. The same principle applies to all function spaces described below.

Let $f$ be a function which is analytic in a neighbourhood of the origin. The $k$th Maclaurin coefficient of $f$ is denoted by $\hat f(k)$. By $\DC_a(\EC)$, we denote the space of $\EC$-valued analytic functions whose Maclaurin coefficients $(\hat f(k))_{k=0}^\infty$ decay faster than any power of $k$. These are precisely the analytic functions on $\DB$ which extend to smooth functions on $\clos{\DB}$.

The \textit{Hardy space} $H^2(\EC)$ consists of all functions $f:z\mapsto \sum_{k=0}^\infty \hat f(k)z^k$, where $\hat f(k)\in\EC$, and $\|f\|_{H^2(\EC)}^2:=\sum_{k=0}^\infty \|\hat f(k)\|_\EC^2<\infty$. Using a standard radius of convergence formula, functions in $H^2(\EC)$ are seen to be analytic on $\DB$. The subspaces $\PC_a(\EC)$ and $\DC_a(\EC)$ are dense in $H^2(\EC)$. A direct application of the Cauchy-Schwarz inequality shows that point evaluations are bounded operators from $H^2(\EC)$ to $\EC$. Specifically,
\begin{align}\label{eq:ReproducingProperty}
\|f(z)\|\le \frac{\|f\|_{H^2(\EC)}}{\left(1-|z|^2\right)^{1/2}},\quad z\in\DB,\ f\in H^2(\EC).
\end{align}
If $f\in H^2(\EC)$, then the radial boundary value $f(\zeta):=\lim_{r\to 1^-}f(r\zeta)$ exists for $\lambda$-a.e. $\zeta\in\TB$. The radial boundary value function satisfies
\begin{align}\label{eq:NormOfBoundaryValues}
\|f\|_{H^2(\EC)}^2=\frac{1}{2\pi}\int_{\TB}\|f(\zeta)\|^2\diff\lambda (\zeta).
\end{align}
In the case where $\EC=\CB$, these facts will be included in any reasonable introduction to Hardy spaces. For general $\EC$, we refer to \cite[Chapter~III]{Nikolski2002:OperatorsFunctionsAndSystems:AnEasyReading.Vol.I}, or \cite[Chapter~4]{Rosenblum-Rovnyak1985:HardyClassesAndOperatorTheory}.

Given an $\LC_+$-valued measure $\mu$, and an analytic function $f:\DB\to\EC$, we define the corresponding \textit{Dirichlet integral}
\begin{align*}
\DC_{\mu}(f):=\frac{1}{\pi}\int_{\DB}\left\langle P_\mu(z)f'(z),f'(z)\right\rangle \diff A(z).
\end{align*}
The integrand is non-negative, so the integral is well-defined. For $z=\rho\zeta$, $\zeta\in\TB$, the power series expansions of $P_\mu$ and $f$ yield
\begin{align*}
\left\langle P_\mu(z)f'(z),f'(z)\right\rangle = \sum_{n=-\infty}^\infty\sum_{k,l=0}^\infty kl\left\langle \hat \mu(n)\hat f(k),\hat f(l)\right\rangle \rho^{|n|+k+l-2}\zeta^{n+k-l}.
\end{align*}
Integrating this identity with respect to $\rho\diff\rho \diff\lambda(\zeta)$, over $0<\rho<r<1$, and $\zeta\in\TB$, gives us the formula
\begin{multline*}
\frac{1}{\pi}\int_{r\DB}\left\langle P_\mu(z)f'(z),f'(z)\right\rangle\diff A(z) 
\\
= 
\sum_{k,l=0}^\infty \min(k,l)r^{2\max(k,l)}\left\langle \hat{\mu}(l-k) \hat f(k),\hat f(l)\right\rangle.
\end{multline*}
By monotone convergence,
\begin{align}\label{eq:DirichletIntegralSeriesRepresentation}
\DC_{\mu}(f)=\lim_{r\to 1^-}\sum_{k,l=0}^\infty \min(k,l)r^{2\max(k,l)}\left\langle \hat{\mu}(l-k) \hat f(k),\hat f(l)\right\rangle.
\end{align}
If the resulting series is absolutely convergent (say if $f\in\DC_a(\EC)$), then we may of course replace the $\lim_{r\to 1^-}$ with an evaluation at $r=1$.

We define the \textit{harmonically weighted Dirichlet space} $\DC_{\mu}^2(\EC)$ as the space of functions $f\in H^2(\EC)$ for which $\DC_{\mu}(f)<\infty$. In particular, $\DC_a(\EC)\subseteq \DC_{\mu}^2(\EC)$ by \eqref{eq:DirichletIntegralSeriesRepresentation}. We equip $\DC_{\mu}^2(\EC)$ with the norm $\|\cdot\|_{\DC_{\mu}^2(\EC)}$ given by
\begin{align*}
\|f\|_{\DC_{\mu}^2(\EC)}^2:=\|f\|_{H^2(\EC)}^2+\DC_{\mu}(f).
\end{align*}
If $f,g\in\DC _\mu^2(\EC)$, then we define the sesquilinear Dirichlet integral
\begin{align*}
\DC_{\mu}(f,g):=\frac{1}{\pi}\int_{\DB}\left\langle P_\mu(z)f'(z),g'(z)\right\rangle \diff A(z),
\end{align*}
which is finite by the Cauchy--Schwarz inequality.

\begin{prop}[{\cite[Corollary 3.1]{Olofsson2004:AVonNeumann-WoldDecompositionOfTwo-Isometries}}]\label{prop:DensityOfPolynomials}
	Let $\mu$ be an $\LC_+$-valued measure. Then $\PC_a(\EC)$ is dense in the corresponding Dirichlet space $\DC_{\mu}^2(\EC)$.
\end{prop}

\begin{prop}\label{prop:DefectOperatorFormula}
	Let $\mu$ be an $\LC_+$-valued measure, and $f\in\DC_a(\EC)$. Then 
	\begin{align*}
	\left\langle \beta_{1}(M_z)f,f\right\rangle_{\DC_{\mu}^2(\EC)}=\frac{1}{2\pi}\int_{\TB}\left\langle\diff\mu\,  f,f\right\rangle. 
	\end{align*}
\end{prop}
\begin{proof}
	By \eqref{eq:NormOfBoundaryValues}, $M_z:H^2(\EC)\to H^2(\EC)$ is an isometry. Hence,
	\[
	\left\langle\beta_{1}(M_z)f,f\right\rangle_{\DC_{\mu}^2(\EC)}=\DC_{\mu}^2(M_zf)-\DC_{\mu}^2(f).
	\]
	By \eqref{eq:DirichletIntegralSeriesRepresentation},
	\begin{align*}
	\DC_{\mu}^2(M_zf)-\DC_{\mu}^2(f)
	= {}&
	\sum_{k,l=0}^\infty \min(k,l)\left\langle\hat{\mu}(l-k)\hat f(k-1),\hat f(l-1)\right\rangle  
	\\
	&-
	\sum_{k,l=0}^\infty \min(k,l)\left\langle\hat{\mu}(l-k)\hat f(k),\hat f(l)\right\rangle 
	\\
	= {}&
	\sum_{k,l=0}^\infty \left\langle\hat{\mu}(l-k)\hat f(k),\hat f(l)\right\rangle.
	\end{align*}
	The above right-hand side equals $\frac{1}{2\pi}\int_{\TB}\left\langle\diff\mu\,  f,f\right\rangle$ by definition. 
\end{proof}

Given an analytic function $f:\DB\to\CB$, and $\zeta\in\TB$, we define the corresponding \textit{local Dirichlet integral}
\begin{align}\label{eq:DefinitionOfLocalDirichletIntegral}
\DC_{\zeta}(f):=\frac{1}{\pi}\int_{\DB}|f'(z)|^2\frac{1-|z|^2}{|\zeta-z|^2}\diff A(z).
\end{align}
This is a convenient shorthand for the Dirichlet integral $\DC_{\delta_\zeta}(f)$, where $\delta_\zeta$ denotes a (scalar) unital point mass at $\zeta$. If $\mu$ is a positive scalar-valued measure, then it is immediate from Fubini's theorem that
\begin{align}\label{eq:FubiniForPositiveMeasures}
\DC_\mu(f)=\frac{1}{2\pi}\int_{\TB}\DC_{\zeta}(f)\diff\mu(\zeta).
\end{align}
In particular, if $\DC_\mu(f)<\infty$, then $\DC_{\zeta}(f)<\infty $ for $\mu$-a.e. $\zeta\in\TB$.

A useful tool for calculating local Dirichlet integrals is the so-called \textit{local Douglas formula}. The proof, and a slightly more general version of the statement, can be found in \cite[Chapter~7.2]{ElFallah-Kellay-Mashreghi-Ransford2014:APrimerOnTheDirichletSpace}:
\begin{thm}\label{thm:LocalDouglasFormula}
	Let $f\in H^2$, and $\zeta\in\TB$. If $\DC_{\zeta}(f)<\infty$, then the radial limit $f(\zeta)$ exists. Moreover, if we let $F(z)=\frac{f(z)-f(\zeta)}{z-\zeta}$, then $\DC_{\zeta}(f)=\|F\|_{H^2}^2$.
\end{thm}

Recall that a function $\theta\in H^2$ is called \textit{inner} if $|\theta(\zeta)|=1$ for $\lambda$-a.e. $\zeta\in\TB$. If $\theta$ is inner and $f\in\DC_a$, then
\begin{align}\label{eq:LocalDirichletFormula}
\DC_{\zeta}(\theta f)=\DC_{\zeta}(f)+|f(\zeta)|^2\DC_{\zeta}(\theta),
\end{align}
e.g. \cite[Theorem~7.6.1]{ElFallah-Kellay-Mashreghi-Ransford2014:APrimerOnTheDirichletSpace}. For inner functions, the local Dirichlet integral $\DC_{\zeta}(\theta)$ may be computed as
\begin{align}\label{eq:LocalDirichletInner}
\DC_{\zeta}(\theta)=\lim_{r\to 1^-}\frac{1-|\theta(r\zeta)|^2}{1-r^2},
\end{align}
e.g. \cite[Theorem~7.6.5]{ElFallah-Kellay-Mashreghi-Ransford2014:APrimerOnTheDirichletSpace}. This formula is valid whether or not $\DC_{\zeta}(\theta)<\infty$. If $\theta$ is analytic in a neighbourhood of $\zeta$, then \eqref{eq:LocalDirichletInner} yields $\DC_{\zeta}(\theta)=|\theta'(\zeta)|$.

In the general case of $\LC_+$-valued measures, the concept of a local Dirichlet integral does not appear to be well studied. We derive the following substitute for \eqref{eq:FubiniForPositiveMeasures}.

\begin{lem}\label{lemma:FubiniForDirichletIntegrals}
	Let $\mu$ be an $\LC_+$-valued measure, and $x,y\in\EC$. If $f\in\DC_{\mu_{x,x}}^2$ and $g\in\DC_{\mu_{y,y}}^2$, then the integral
	\begin{align*}
	\DC_{\mu_{x,y}}(f,g)=\frac{1}{2\pi^2}\int_{\DB}\int_{\TB}f'(z)\conj{g'(z)}\frac{1-|z|^2}{|\zeta-z|^2}\diff\mu_{x,y}(\zeta)\diff A(z)
	\end{align*}
	is absolutely convergent, with $|\DC_{\mu_{x,y}}(f,g)|\le \|f\|_{\DC_{\mu_{x,x}}^2}\|g\|_{\DC_{\mu_{y,y}}^2}$. In particular,
	\begin{align*}
	\DC_{\mu_{x,y}}(f,g)=\frac{1}{2\pi}\int_{\TB}\DC_\zeta(f,g)\diff\mu_{x,y}(\zeta).
	\end{align*}
\end{lem}
\begin{proof}
	We need to show that
	\begin{align*}
	\int_{\DB}\int_{\TB}|f'(z)\conj{g'(z)}|\frac{1-|z|^2}{|\zeta-z|^2}\diff|\mu_{x,y}|(\zeta)\diff A(z)<\infty.
	\end{align*}
	The inner integral equals $|f'(z)\conj{g'(z)}|P_{|\mu_{x,y}|}(z)$.	By an application of Lemma~\ref{lemma:VersionOfCauchySchwarz}, we obtain
	\[
	P_{|\mu_{x,y}|}(z)\le P_{\mu_{x,x}}(z)^{1/2}P_{\mu_{y,y}}(z)^{1/2}.
	\]
	The remainder of the statement follows from the Cauchy--Schwarz inequality, and Fubini's theorem.
\end{proof}

\subsection{Analytic operators}\label{SubSec:AnalyticOperators}

An operator $T\in\LC$ is called \textit{analytic} if $\cap_{n\ge 0}T^n\HC=\{0\}$. It is clear that an analytic operator cannot have any non-zero eigenvalues. In particular, an analytic $m$-isometry does not have eigenvalues, since $\sigma_{ap}(T)\subseteq \TB$.

A good reason to study analytic $2$-isometries is the existence of a so-called \textit{Wold decomposition}: For $T\in\LC$, we let $\EC=\HC\ominus T\HC$. The dimension of $\EC$ is called the \textit{multiplicity} of $T$. Furthermore, define the spaces $\HC_u=\cap_{n\ge 0}T^n\HC$, and $\HC_a=\bigvee_{n\ge 0}T^n\EC$. The following is a special case of \cite[Theorem 3.6]{Shimorin2001:Wold-TypeDecompositionsAndWanderingSubspacesForOperatorsCloseToIsometries}:
\begin{thm}\label{thm:WoldDecomposition}
	If $T\in\LC$ is a $2$-isometry, then $\HC=\HC_u\oplus\HC_a$. Moreover, the spaces $\HC_u$ and $\HC_a$ are invariant under $T$, $T_u:=T|_{\HC_u}$ is unitary, and $T_a:=T|_{\HC_a}$ is analytic. 
\end{thm}

The class of analytic $2$-isometries with multiplicity $1$ can be completely described in terms of $M_z$, multiplication by the function $z\mapsto z$, acting on harmonically weighted Dirichlet spaces \cite{Richter1991:ARepresentationTheoremForCyclicAnalyticTwo-Isometries}. This result was later generalized to arbitrary multiplicity by Olofsson \cite{Olofsson2004:AVonNeumann-WoldDecompositionOfTwo-Isometries}:

\begin{thm}\label{thm:RichterOlofssonFunctionalModel}
	Let $\mu$ be an $\LC_+(\EC)$-valued measure. Then $M_z$ acts as an analytic $2$-isometry on the space $\DC_{\mu}^2(\EC)$. 
	
	Conversely, suppose that $T\in\LC(\HC)$ is an analytic $2$-isometry, and let $\EC=\HC\ominus T\HC$. Then there exists an $\LC_+(\EC)$-valued measure $\mu$, and a unitary map $V:\HC\to\DC_{\mu}^2(\EC)$, such that $T=V^*M_zV$. 
	
	The above correspondence is essentially one-to-one; the operators
	\[
	M_z:\DC_{\mu_j}^2(\EC_j)\to \DC_{\mu_j}^2(\EC_j),\quad j\in \{1,2\},
	\] 
	are unitarily equivalent if and only if there exists a unitary map $U:\EC_1\to\EC_2$ such that $\mu_1(E)=U^*\mu_2(E)U$ whenever $E\in\mathfrak{S}$.
\end{thm}

\section{$m$-isometric semigroups}\label{sec:m-IsometricSemigroups}

By an \textit{$m$-isometric semigroup} we mean a $C_0$-semigroup $(T_t)_{t\ge 0}$ such that each $T_t$ is an $m$-isometry. In the transition from individual operators to $C_0$-semigroups, the following two lemmas are useful:

\begin{lem}\label{lemma:DerivativeOfPolarizedOrbitMap}
	Let $(T_t)_{t\ge 0}$ be a $C_0$-semigroup with generator $A$, and $m\in\ZB_{\ge 0}$. If $y_1,y_2\in D(A^m)$, and $f_{y_1,y_2}:[0,\infty)\ni t\mapsto \left\langle T_ty_1,T_ty_2\right\rangle$, then $f_{y_1,y_2}$ is $m$ times differentiable, and its $m$th derivative $f_{y_1,y_2}^{(m)}$ is given by
	\begin{align}\label{eq:DerivativeOfPolarizedOrbitMap}
	f_{y_1,y_2}^{(m)}(t)=\alpha_m^A( T_ty_1,T_ty_2),\quad t\ge 0,
	\end{align}
	where $\alpha_m^A$ is defined by \eqref{eq:DefinitionOfAlphaM}. In particular, $f_{y_1,y_2}^{(m)}$ is continuous.
\end{lem}
\begin{proof}
	This is trivial for $m=0$. For $m=1$, if $y_1,y_2\in D(A)$, $t\ge 0$, then
	\[
	f_{y_1,y_2}'(t)=f_{Ay_1,y_2}(t)+f_{y_1,Ay_2}(t),
	\]
	by the product rule. By induction one obtains
	\[
	f_{y_1,y_2}^{(m)}(t)=\sum_{k=0}^m\binom{m}{k}f_{A^ky_1,A^{m-k}y_2}(t),
	\]
	provided that $y_1,y_2\in D(A^m)$. This is equivalent to \eqref{eq:DerivativeOfPolarizedOrbitMap}. 
\end{proof}

\begin{lem}\label{lemma:HigherOrderDifferenceQuotient}
	Let $f:[0,\infty)\to\CB$ be a function whose $m$th derivative $f^{(m)}$ is continuous. For $h>0$, it then holds that
	\begin{align*}
	\sum_{k=0}^m(-1)^{m-k}\binom{m}{k}f(kh)=\int_{s_1=0}^h\ldots \int_{s_m=0}^hf^{(m)}(s_1+\ldots+s_m)\diff s_1\ldots \diff s_m.
	\end{align*}
\end{lem}
\begin{proof}
	Use the fundamental theorem of calculus, \eqref{eq:BinomialIdentity}, and induction over $m$.
\end{proof}

As a first application of these identities, we prove a semigroup analogue of \eqref{eq:PolynomialGrowthDiscreteTime}. This is more precise than condition $(ii)$ of  \cite[Theorem~2.1]{Bermudez-Bonilla-Zaway2019:C0-SemigroupsOfM-IsometriesOnHilbertSpaces}:

\begin{thm}\label{thm:PolynomialGrowthContinuousTime}
	Let $(T_t)_{t\ge 0}$ be an $m$-isometric semigroup with generator $A$. Then the quadratic forms $(\alpha_j^A)_{j=0}^{m-1}$, densely defined by \eqref{eq:DefinitionOfAlphaM}, are bounded. Moreover,
	\begin{align*}
	\|T_tx\|^2=\sum_{j=0}^{m-1}\frac{t^j}{j!}\alpha_j^A(x),\quad x\in\HC, t\ge 0.
	\end{align*}
\end{thm}
\begin{proof}
	For $y\in D(A^j)$, let $f(t)=\|T_ty\|^2$. By Lemma~\ref{lemma:DerivativeOfPolarizedOrbitMap}, $f^{(j)}(t)= \alpha_j^A(T_ty)$. Moreover, this is a continuous function, and by Lemma~\ref{lemma:HigherOrderDifferenceQuotient},
	\begin{align}\label{eq:HigherOrderDefectAsIntegral}
	\left\langle \beta_{j}(T_{t})y,y\right\rangle =\int_{s_1=0}^{t}\ldots\int_{s_j=0}^{t} \alpha_j^A(T_{s_1+\ldots+s_j}y)\diff s_1\ldots \diff s_j.
	\end{align}	
	This identity implies
	\begin{align}\label{eq:HigherOrderDefectAsIntegral'}
	\lim_{t\to 0^+}\frac{1}{t^j}\left\langle \beta_{j}(T_{t})y,y\right\rangle= \alpha_{j}^A(y).
	\end{align}
	By \eqref{eq:PolynomialGrowthDiscreteTime}, it holds for any $k\in\ZB_{\ge 0}$ that
	\begin{align*}
	\|T_{t/k}^{k}y\|^2
	=
	\sum_{j=0}^{m-1}\binom{k}{j}\left\langle \beta_{j}(T_{t/k})y,y\right\rangle
	=
	\sum_{j=0}^{m-1}\left(\frac{t}{k}\right)^j\binom{k}{j}\frac{\left\langle \beta_{j}(T_{t/k})y,y\right\rangle}{\left(t/k\right)^j}.
	\end{align*}
	If $y\in D(A^{m-1})$, then we may let $k\to\infty$, in order to obtain
	\begin{align*}
	\|T_ty\|^2=\sum_{j=0}^{m-1}\frac{t^j}{j!}\alpha_j^A(y).
	\end{align*}	
	We want to extend the above identity to $y\in\HC$. For this it is sufficient to prove that each $\alpha_j^A$ is bounded. Given $m$ distinct times $(t_k)_{k=0}^{m-1}$, the numbers $(\|T_{t_k}y\|^2)_{k=0}^{m-1}$ uniquely determine $ \alpha_j^A(y)$. By linear algebra, it even holds that $ \alpha_j^A(y)=\sum_{k=0}^{m-1}a_{jk}\|T_{t_k}y\|^2$, where the numbers $(a_{jk})$ do not depend on $y$. It follows that each $ \alpha_j^A$ is a bounded quadratic form.
\end{proof}

Together with \eqref{eq:LaplaceTransformOfSemigroup}, Theorem~\ref{thm:PolynomialGrowthContinuousTime} implies the well-known result that if $A$ is the generator of an $m$-isometric semigroup, then $\sigma(A)\subseteq\{z\in \CB;\ \Re z\le 0\}$. A novel result is the following:

\begin{cor}\label{cor:NumericalRangeOfGenerator}
	Let $(T_t)_{t\ge 0}$ be an $m$-isometric semigroup with generator $A$. Then, there exists $a\in\RB$, $b\in (a,\infty)$, such that
	\[
	W(A)\subseteq \left\{z\in\CB;\,  a\le \Re z\le b \right\}.
	\]
\end{cor}
\begin{proof}
	By Theorem~\ref{thm:PolynomialGrowthContinuousTime}, $\alpha_1^A(y)=2\Re \left\langle Ay,y\right\rangle$ defines a bounded quadratic form on $\HC$.
\end{proof}

\begin{cor}\label{cor:QuasiContractive}
	Let $(T_t)_{t\ge 0}$ be an $m$-isometric semigroup. Then there exists $w\ge 0$ such that
	\begin{align*}
	\|T_t\|\le e^{wt},\quad t\ge 0,
	\end{align*}
	i.e. $(T_t)_{t\ge 0}$ is quasicontractive with parameter $w$.
\end{cor}
\begin{proof}
	Since $\alpha_0^A=I$, Theorem~\ref{thm:PolynomialGrowthContinuousTime} yields that
	\[
	\|T_t\|^2\le 1+tp(t),
	\]
	for some polynomial $p$. It is clear that the right-hand side is dominated by $e^{wt}$ for some $w\ge 0$.
\end{proof}

\subsection{Proof of Theorem~\ref{thm:MainGen}}\label{subsec:ProofOfMainGen}

Suppose that $A$ is the generator of an $m$-isometric semigroup $(T_t)_{t\ge 0}$. Property $(i)$ holds for any generator. By Corollary~\ref{cor:QuasiContractive}, $(T_t)_{t\ge 0}$ is quasicontractive for some parameter $w\ge 0$. Properties $(ii)$ and $(iii)$ are implied by the Lumer--Phillips theorem (Theorem~\ref{thm:Lumer--Phillips}). Since each $T_t$ is $m$-isometric, \eqref{eq:HigherOrderDefectAsIntegral'} implies that $\alpha_m^A$ vanishes, i.e. $A$ is $m$-skew-symmetric. This proves that the conditions $(i)-(iv)$ are necessary.

Suppose on the other hand that $A$ satisfies $(i)-(iv)$. The first three conditions, together with the Lumer--Phillips theorem, imply that $A$ generates a $C_0$-semigroup $(T_t)_{t\ge 0}$, quasicontractive with parameter $w\ge 0$. The assumption that $A$ is $m$-skew-symmetric implies that each $T_t$ is $m$-isometric, by \eqref{eq:HigherOrderDefectAsIntegral}.

Finally, if $(i)-(iv)$ are fulfilled, so that $A$ generates an $m$-isometric $C_0$-semigroup $(T(t))_{t\ge 0}$, Theorem~\ref{thm:PolynomialGrowthContinuousTime} together with \eqref{eq:LaplaceTransformOfSemigroup} implies that $\lambda-A:D(A)\to\HC$ is invertible whenever $\lambda >0$. 

\subsection{Proof of Theorem~\ref{thm:MainCoGen}}
Let $(T_t)_{t\ge 0}$ be an $m$-isometric semigroup with generator $A$. By Theorem~\ref{thm:MainGen}, $A$ is $m$-skew-symmetric, and $(\lambda-A)^{-1}\in\LC$ for any $\lambda>0$. In particular, $1\in\rho(A)$. For $T=(A+I)(A-I)^{-1}$, Lemma~\ref{lemma:M-IsometryToM-SkewSymmetry} implies that $T$ is $m$-isometric, and that each $y\in D(A)$ can be written as $\frac{1}{2}(T-I)x$ for precisely one $x\in\HC$. Under this correspondence $Ay=\frac{1}{2}(T+I)x$. It follows that \eqref{eq:QCCondGen} holds if and only if
\[
\|Tx\|^2-\|x\|^2=\Re \left\langle (T+I)x,(T-I)x\right\rangle \le w \|(T-I)x\|^2,\quad x\in\HC,
\]
i.e. \eqref{eq:QCCondCoGen} holds.

\section{$2$-isometric cogenerators}\label{sec:2-isometricCogenerators}

The purpose of this section is to prove Theorem \ref{thm:MainConverse}, which essentially amounts to proving Theorem \ref{thm:MainMeasures}. 

We recall the hypothesis of Theorem~\ref{thm:MainConverse}: $T\in\LC$ is a $2$-isometry, $T-I$ is injective, and there exists $w\ge 0$ such that
\begin{align}\label{eq:QCCondCoGenRestate}
\left\langle \beta_1(T)x,x\right\rangle = \|Tx\|^2-\|x\|^2 \le w\|(T-I)x\|^2,\quad x\in \HC.
\end{align}
The assertion that we wish to prove is that $T$ is the cogenerator of a $C_0$-semigroup $(T_t)_{t\ge 0}$, and that this is quasicontractive with parameter $w$.

The first step is a reduction to the case of analytic operators. By the Wold-decomposition (Theorem \ref{thm:WoldDecomposition}), $T=T_u\oplus T_a$, where $T_u$ is unitary and $T_a$ is analytic. Since $1\notin \sigma_p(T)$, we know that $1\notin\sigma_p(T_u)$. By \cite[Chapter~III, Section~8]{Sz.-Nagy-Foias-Bercovici-Kerchy2010:HarmonicAnalysisOfOperatorsOnHilbertSpace}, $T_u$ is the cogenerator of a $C_0$-semigroup of unitary operators on $\HC_u$. This is clearly quasicontractive for any $w\ge 0$. It therefore suffices to show that $T_a$ is the cogenerator of a quasicontractive $C_0$-semigroup on $\HC_a$.

By orthogonality and invariance of the subspaces $\HC_u$ and $\HC_a$, and the fact that $T_u$ is unitary, Theorem \ref{thm:WoldDecomposition} further implies that \eqref{eq:QCCondCoGenRestate} is equivalent to
\[
\|T_ax_a\|^2-\|x_a\|^2\le w\|(T_u-I_u)x_u\|^2+w\|(T_a-I_a)x_a\|^2,\quad (x_u,x_a)\in\HC_u\oplus\HC_a.
\]
Consequently, $T$ satisfies \eqref{eq:QCCondCoGenRestate} if and only if $T_a$ does. Together with the previous paragraph, this reduces the proof of Theorem \ref{thm:MainConverse} to the case where $T$ is an analytic operator.

With the additional hypothesis that $T$ is analytic, let $\EC=\HC\ominus T\HC$. By Theorem \ref{thm:RichterOlofssonFunctionalModel}, $T=V^*M_zV$, where $V:\HC\to\DC_{\mu}^2(\EC)$ is a unitary map, and $\mu$ is some $\LC_+$-valued measure. It then holds that $\beta_1(T)=V^*\beta_1(M_z)V$, and \eqref{eq:QCCondCoGenRestate} becomes
\begin{align*}
\left\langle\beta_1(M_z)f,f\right\rangle_{\DC_{\mu}^2(\EC)} \le w\|(I-M_z)f\|_{\DC_{\mu}^2(\EC)}^2,\quad f\in \DC_{\mu}^2(\EC).
\end{align*}
By Proposition~\ref{prop:DensityOfPolynomials}, it is sufficient to verify this for $f\in\PC_a(\EC)$, and by Proposition \ref{prop:DefectOperatorFormula}, $T$ satisfies \eqref{eq:QCCondCoGenRestate} if and only if 
\begin{align*}
\frac{1}{2\pi}\int_{\TB}\left\langle \diff\mu\, f,f\right\rangle \le w\|(I-M_z)f\|_{\DC_{\mu}^2(\EC)}^2,\quad f\in \PC_a(\EC).
\end{align*}
If we for a moment assume the validity of Theorem \ref{thm:MainMeasures}, then the above condition implies that $(M_{\phi_t})_{t\ge 0}$ is a $C_0$-semigroup on $\DC_\mu^2(\EC)$, and that this is quasicontractive with parameter $w$. Let $(T_t)_{t\ge 0}=(V^*M_{\phi_t}V)_{t\ge 0}$. This defines a $C_0$-semigroup on $\HC$, it's quasicontractive with parameter $w$, and its cogenerator is given by $V^*M_zV=T$.

\subsection{Proof of Theorem \ref{thm:MainMeasures}}

Let us recall the statement of Theorem~\ref{thm:MainMeasures}: If $\mu$ is an $\LC_+$-valued measure on $\TB$, then the following are equivalent:
\begin{enumerate}[$(i)$]
	\item For every $t\ge 0$, $M_{\phi_t}\in\LC(\DC_\mu^2(\EC))$, and the family $(M_{\phi_t})_{t\ge0}$ is a $C_0$-semigroup.
	\item There exists $w_1\ge 0$ such that 
	\begin{align*}
	\frac{1}{2\pi}\int_{\TB}\left\langle \diff\mu\, f,f\right\rangle
	\le
	w_1
	\|(I-M_z)f\|_{\DC_\mu^2(\EC)}^2,\quad f\in\PC_a(\EC).
	\end{align*}
	\item The set function $\tilde \mu: E\mapsto \frac{1}{2\pi}\int_{E}\frac{\diff\mu(\zeta)}{|1-\zeta|^2}$ is an $\LC_+$-valued measure, and there exists $w_2\ge 0$ such that 
	\begin{align*}
	\frac{1}{2\pi}\int_{\TB}\left\langle \diff\tilde{\mu}\, f,f\right\rangle
	\le
	w_2
	\|f\|_{\DC_\mu^2(\EC)}^2,\quad f\in\PC_a(\EC).
	\end{align*}
\end{enumerate}
If one (hence all) of the above conditions is satisfied, then the $C_0$-semigroup $(M_{\phi_t})_{t\ge0}$ is $2$-isometric, has cogenerator $M_z$, and is quasicontractive for some parameter $w\ge 0$. Moreover, the optimal (smallest) values of $w$, $w_1$, and $w_2$ coincide.

In the case $\EC=\CB$, there is a quite direct proof that $(i)\Leftrightarrow(iii)$. In the general case, we essentially use the same ideas, although they are somewhat obscured by technicalities. For this reason, we begin with a preliminary discussion of the case $\EC=\CB$. In general, the qualitative assertions that $(M_{\phi_t})_{t\ge 0}$ is $2$-isometric and quasicontractive are fairly immediate from $(i)$. The argument also shows that $(i)\Rightarrow(ii)$. The implications $(ii)\Rightarrow(iii)\Rightarrow(i)$ require a bit more work. It will be evident that the optimal values of $w$, $w_1$, and $w_2$ coincide.

\begin{proof}[$\EC=\CB\,$\textnormal{:}]\label{p:ScalarCase}
	Note that $\phi_t:z\mapsto \exp(t(z+1)/(z-1))$ is inner, i.e. $\phi_t\in H^2$, and $|\phi_t(\zeta)|=1$ for $\lambda$-a.e. $\zeta\in\TB$. By \eqref{eq:LocalDirichletInner}, 
	\[
	\DC_{\zeta}(\phi_t)
	=
	\frac{2t}{|1-\zeta|^2},
	\]
	with the interpretation $\frac{1}{|1-1|^2}=\infty$. If we assume $(i)$, then $\DC_{\mu}(\phi_t)<\infty$, so $\frac{2t}{|1-\zeta|^2}$ is finite for $\mu$-a.e. $\zeta\in\TB$. In particular, $\mu(\{1\})=0$. On the other hand, if we assume $(iii)$, then again $\mu(\{1\})=0$. We may therefore add this condition to our hypothesis.
	
	By \eqref{eq:LocalDirichletFormula}
	\[
	\DC_{\zeta}(\phi_tf)
	=
	\DC_{\zeta}(f)+|f(\zeta)|^2\frac{2t}{|1-\zeta|^2},
	\]
	and this holds for $\mu$-a.e. $\zeta\in\TB$. Using Fubini's theorem \eqref{eq:FubiniForPositiveMeasures}, the above formula implies that
	\begin{align}\label{eq:11111}
	\DC_\mu(\phi_tf)=\DC_\mu(f)+\frac{t}{\pi}\int_{\TB}\frac{|f(\zeta)|^2}{|1-\zeta|^2}\diff\mu(\zeta).
	\end{align}
	Adding $\|\phi_t f\|_{H^2}^2=\|f\|_{H^2}^2$ to the above, we obtain
	\[
	\|\phi_t f\|_{\DC_\mu^2}^2
	=
	\|f\|_{\DC_\mu^2}^2+\frac{t}{\pi}\int_{\TB}\frac{|f(\zeta)|^2}{|1-\zeta|^2}\diff\mu(\zeta).
	\]
	From this, we conclude that $M_{\phi_t}:\DC_\mu^2\to\DC_\mu^2$ is bounded for all $t\ge0$ if and only if
	\[
	\frac{1}{\pi}\int_{\TB}\frac{|f(\zeta)|^2}{|1-\zeta|^2}\diff\mu(\zeta)\lesssim \|f\|_{\DC_\mu^2}^2,\quad f\in\PC_a.
	\]
	For the verification that, under such circumstances, $(M_{\phi_t})_{t\ge 0}$ is indeed a $C_0$-semigroup, we refer to the general case below.
\end{proof} 

\begin{rem}
	The above argument shows that for $\EC=\CB$, condition $(i)$ may be weakened to:
	\begin{enumerate}[$(i')$]
		\item There exists $t>0$ such that $M_{\phi_t}\in\LC(\DC_\mu^2)$.
	\end{enumerate}
	The author has not found a proof that the same phenomenon occurs in the general case.
\end{rem}

\begin{proof}[$(i)\Rightarrow(ii)\,$\textnormal{:}]
	Assume that $(M_{\phi_t})_{t\ge 0}\subset \LC(\DC_{\mu}^2(\EC))$ is a $C_0$-semigroup. A simple calculation shows that $M_z$ is the cogenerator of $(M_{\phi_t})_{t\ge 0}$.  $M_z$ is $2$-isometric, and by Lemma~\ref{lemma:M-IsometryToM-SkewSymmetry}, the corresponding generator is $2$-skew-symmetric. By Theorem~\ref{thm:MainGen}, $(M_{\phi_t})_{t\ge 0}$ is $2$-isometric. By Corollary~\ref{cor:QuasiContractive}, $(M_{\phi_t})_{t\ge 0}$ is quasicontractive for some $w\ge 0$. For any such $w$, Theorem~\ref{thm:MainCoGen} implies that
	\[
	\left\langle \beta_1(M_z)f,f\right\rangle_{\DC_{\mu}^2(\EC)} \le w\|(M_z-I)f\|_{\DC_{\mu}^2(\EC)}^2,\quad f\in \DC_{\mu}^2(\EC).
	\]
	In particular, the above inequality is satisfied for $f\in\PC_a(\EC)$. By Proposition~\ref{prop:DefectOperatorFormula}, condition $(ii)$ holds with $w_1=w$.
\end{proof}

\begin{proof}[$(ii)\Rightarrow(iii)\,$\textnormal{:}] 
	Consider a fixed $w_1\ge 0$, and assume that
	\[
	\frac{1}{2\pi}\int_{\TB}\left\langle \diff\mu\, f,f\right\rangle \le w_1\|(I-M_z)f\|_{\DC_\mu^2(\EC)}^2,\quad f\in \PC_a(\EC).
	\]
	For $f\in\DC_a(\EC)$, we let $f_N(z)=\sum_{k=0}^N\hat f(k)z^k$. Clearly, $\lim_{N\to\infty }\|f_N\|_{H^2(\EC)}^2=\|f\|_{H^2(\EC)}^2$. Applying the above inequality to $f_N\in\PC_a(\EC)$, and letting $N\to\infty$, the formulas \eqref{eq:DefinitionOfQuadraticIntegral} and \eqref{eq:DirichletIntegralSeriesRepresentation} imply
	\begin{align}\label{eq:WeakEmbeddingCondition}
	\frac{1}{2\pi}\int_{\TB}\left\langle \diff\mu\, f,f\right\rangle \le w_1\|(I-M_z)f\|_{\DC_\mu^2(\EC)}^2,\quad f\in \DC_a(\EC),
	\end{align}
	We will show that \eqref{eq:WeakEmbeddingCondition} implies
	\begin{align}\label{eq:StrongEmbeddingCondition}
	\frac{1}{2\pi}\int_{\TB}\left\langle \frac{\diff\mu(\zeta)}{|1-\zeta|^2}\, f(\zeta),f(\zeta)\right\rangle \le w_1\|f\|_{\DC_\mu^2(\EC)}^2,\quad f\in \PC_a(\EC).
	\end{align}
	In particular, if condition $(ii)$ holds for some $w_1$, then condition $(iii)$ holds with $w_2=w_1$.
	
	Let $f\in\PC_a(\EC)$. In order to prove \eqref{eq:StrongEmbeddingCondition}, we apply \eqref{eq:WeakEmbeddingCondition} to the function $z\mapsto \frac{f(z)}{1-rz}$, where $0<r<1$. If we let $k_r:z\mapsto\frac{1-z}{1-rz}$, then \eqref{eq:WeakEmbeddingCondition} implies
	\[
	\frac{1}{2\pi}\int_{\TB}\left\langle \diff\mu(\zeta)\, \frac{f(\zeta)}{1-r\zeta},\frac{f(\zeta)}{1-r\zeta}\right\rangle 
	\le w_1\|k_rf\|_{H^2(\EC)}^2+w_1\DC_{\mu}(k_rf).
	\]
	The proof is finished by verifying the three statements
	\begin{align}\label{eq:Claim1}
	\lim_{r\to 1^-}\|k_rf\|_{H^2(\EC)}^2=\|f\|_{H^2(\EC)}^2,
	\end{align}
	\begin{align}\label{eq:Claim2}
	\lim_{r\to 1^-}\frac{1}{2\pi}\int_{\TB}\left\langle \diff\mu(\zeta)\, \frac{f(\zeta)}{1-r\zeta},\frac{f(\zeta)}{1-r\zeta}\right\rangle=\frac{1}{2\pi}\int_{\TB}\left\langle \frac{\diff\mu(\zeta)}{|1-\zeta|^2}\, f(\zeta),f(\zeta)\right\rangle,
	\end{align}
	and
	\begin{align}\label{eq:Claim3}
	\lim_{r\to 1^-}\DC_{\mu}(k_rf)=\DC_{\mu}(f).
	\end{align}
	
	For any $z\in\clos{\DB}\setminus\{1\}$, it holds that $|k_r(z)|\le\frac{2}{1+r}$, and $\lim_{r\to 1^-}k_r(z)=1$. By \eqref{eq:NormOfBoundaryValues}, and dominated convergence, \eqref{eq:Claim1} holds for any $f\in H^2(\EC)$.
	
	For the verification of \eqref{eq:Claim2}, an important step is to show that $\frac{\diff\mu(\zeta)}{|1-\zeta|^2}$ does indeed define a measure, so that the right-hand side is well-defined. We begin with a lemma, based on the local Douglas formula:
	\begin{lem}\label{lemma:Convergence2}
		Let $f\in H^2$, and $\zeta\in\TB$. If $\DC_{\zeta}(f)<\infty$, then
		\begin{align*}
		\DC_{\zeta}(k_rf)\le \frac{8}{(1+r)^2}\DC_{\zeta}(f)+2\frac{1-r}{1+r}\frac{|f(\zeta)|^2}{|1-r\zeta|^2}.
		\end{align*}
		Moreover, if we let $F(z)=\frac{f(z)-f(\zeta)}{z-\zeta}$, then
		\begin{align*}
		\DC_{\zeta}(f-k_rf)
		&\le 
		2\left(\|F-k_rF\|_{H^2}^2+\frac{1-r}{1+r}\frac{|f(\zeta)|^2}{|1-r\zeta|^2}\right).
		\end{align*}
		In particular, if $\zeta\ne 1$, then $\lim_{r\to 1^-}\DC_{\zeta}(k_rf)=\DC_{\zeta}(f)$.
	\end{lem}
	\begin{proof}		
		Theorem~\ref{thm:LocalDouglasFormula} already asserts that $f(\zeta)$ exists, and $\DC_\zeta(f)=\|F\|_{H^2}^2$. Applying the same formula to $k_rf$,
		\begin{align*}
		\DC_{\zeta}(k_rf)
		&=
		\|\frac{k_r(z)f(z)-k_r(\zeta)f(\zeta)}{z-\zeta
		}\|_{H^2}^2
		\\
		&=
		\|k_r(z)\frac{f(z)-f(\zeta)}{z-\zeta}+\frac{k_r(z)-k_r(\zeta)}{z-\zeta}f(\zeta)\|_{H^2}^2
		\\
		&=
		\|k_r(z)F(z)-\frac{1-r}{(1-rz)(1-r\zeta)}f(\zeta)\|_{H^2}^2
		\\
		&\le 
		2\left(\|k_rF\|_{H^2}^2+\|\frac{1-r}{(1-rz)(1-r\zeta)}f(\zeta)\|_{H^2}^2\right)
		\\
		&\le 
		2\left(\frac{4}{(1+r)^2}\|F\|_{H^2}^2+(1-r)^2\frac{|f(\zeta)|^2}{|1-r\zeta|^2}\|\frac{1}{1-rz}\|_{H^2}^2\right)
		\\
		&=
		\frac{8}{(1+r)^2}\DC_{\zeta}(f)+2\frac{1-r}{1+r}\frac{|f(\zeta)|^2}{|1-r\zeta|^2}.
		\end{align*}
		In the last step, we have used geometric summation to compute $\|\frac{1}{1-rz}\|_{H^2}^2=
		\frac{1}{1-r^2}$. This proves the first inequality of the statement. The second inequality follows from a similar calculation.
	\end{proof}
	
	The next lemma essentially states that if \eqref{eq:WeakEmbeddingCondition} holds, then a slightly weaker version of \eqref{eq:StrongEmbeddingCondition} holds for any $f$ with values in a one-dimensional subspace.
	
	\begin{lem}\label{lemma:WeakEmbedding}
		Let $\mu$ be an $\LC_+$-valued measure that satisfies \eqref{eq:WeakEmbeddingCondition}. Then
		\[
		\frac{1}{2\pi}\int_{\TB}\frac{|f(\zeta)|^2}{|1-\zeta|^2}\diff\mu_{x,x}(\zeta)\le w_1\left(\|f\|_{H^2}^2\|x\|^2+2\DC_{\mu_{x,x}}(f)\right),\quad x\in\EC,\ f\in\PC_a.
		\]
		In particular, the set function $E\mapsto\int_{E}\frac{\diff \mu(\zeta)}{|1-\zeta|^2}$ is an $\LC_+$-valued measure, and $\mu(\{1\})=0$. 
	\end{lem}
	\begin{proof}
		Applying \eqref{eq:WeakEmbeddingCondition} to the function $z\mapsto \frac{f(z)}{1-rz}x$ yields
		\begin{align}\label{eq:Auxiliary}
		\frac{1}{2\pi}\int_{\TB} \frac{|f(\zeta)|^2}{|1-r\zeta|^2} \diff\mu_{x,x}(\zeta)\le w_1\left(\|k_rf\|_{H^2}^2\|x\|^2+\DC_{\mu_{x,x}}(k_rf)\right).
		\end{align}
		Using Lemma~\ref{lemma:Convergence2}, together with Fubini's theorem \eqref{eq:FubiniForPositiveMeasures},
		\[
		\DC_{\mu_{x,x}}(k_rf)\le \frac{8}{(1+r)^2}\DC_{\mu_{x,x}}(f)+\frac{1}{\pi}\frac{1-r}{1+r}\int_\TB \frac{|f(\zeta)|^2}{|1-r\zeta|^2}\diff\mu_{x,x}(\zeta).
		\]
		Thus, subtracting $\frac {w_1}{\pi}\frac{1-r}{1+r}\int_\TB \frac{|f(\zeta)|^2}{|1-r\zeta|^2}\diff\mu_{x,x}(\zeta)$ from both sides of \eqref{eq:Auxiliary} yields
		\begin{multline*}
		\left(\frac{1}{2\pi}-\frac {w_1}{\pi}\frac{1-r}{1+r}\right)\int_{\TB} \frac{|f(\zeta)|^2}{|1-r\zeta|^2} \diff\mu_{x,x}(\zeta)
		\\
		\le w_1\left(\|k_rf\|_{H^2}^2\|x\|^2+\frac{8}{(1+r)^2}\DC_{\mu_{x,x}}(f)\right).
		\end{multline*}
		A formal application of Fatou's lemma yields
		\begin{align*}
		\frac{1}{2\pi}\int_{\TB}\frac{|f(\zeta)|^2}{|1-\zeta|^2}\diff\mu_{x,x}(\zeta)
		&=
		\frac{1}{2\pi}\int_{\TB}\liminf_{r\to 1^-}\frac{|f(\zeta)|^2}{|1-r\zeta|^2}\diff\mu_{x,x}(\zeta)
		\\
		&\le
		\liminf_{r\to 1^-}
		\frac{1}{2\pi}\int_{\TB}\frac{|f(\zeta)|^2}{|1-r\zeta|^2}\diff\mu_{x,x}(\zeta)
		\\
		&=
		\liminf_{r\to 1^-}\left(\frac{1}{2\pi}-\frac {w_1}{\pi}\frac{1-r}{1+r}\right)\int_{\TB} \frac{|f(\zeta)|^2}{|1-r\zeta|^2} \diff\mu_{x,x}(\zeta)
		\\
		&\le
		\liminf_{r\to 1^-}w_1\left(\|k_rf\|_{H^2}^2\|x\|^2+\frac{8}{(1+r)^2}\DC_{\mu_{x,x}}(f)\right)
		\\
		&=
		w_1\left(\|f\|_{H^2}^2\|x\|^2+2\DC_{\mu_{x,x}}(f)\right).
		\end{align*}
		In order to justify this, we need to argue that $\frac{|f(\zeta)|^2}{|1-\zeta|^2}=\lim_{r\to 1^-}\frac{|f(\zeta)|^2}{|1-r\zeta|^2}$ for $\mu_{x,x}$-a.e. $\zeta\in\TB$. For the constant function $f\equiv 1$ we have convergence for every $\zeta$, and we obtain
		\[
		\frac{1}{2\pi}\int_{\TB}\frac{1}{|1-\zeta|^2}\diff\mu_{x,x}(\zeta)\le w_1\|x\|^2,\quad x\in\EC.
		\]
		This implies that $\frac{1}{|1-\zeta|^2}$ is finite for $\mu_{x,x}$-a.e. $\zeta\in\TB$. As this holds for every $x\in\EC$, $\mu(\{1\})=0$. This justifies the above argument for any $f\in\PC_a(\EC)$.
		
		As was noted in Subsection \ref{subsec:OperatorMeasures}, the above inequality also implies that $E\mapsto\int_{E}\frac{\diff \mu(\zeta)}{|1-\zeta|^2}$ is an $\LC_+$-valued measure. 
	\end{proof}
	
	By hypothesis \eqref{eq:WeakEmbeddingCondition} holds, so Lemma~\ref{lemma:WeakEmbedding} allows us to assume that $\mu(\{1\})=0$. 
	
	We now prove \eqref{eq:Claim2}. By Lemma~\ref{lemma:WeakEmbedding}, the right-hand side is well-defined. For $f\in\PC_a(\EC)$, let $\{e_n\}\subset\EC$ be an orthonormal basis of a finite-dimensional subspace containing the range of $f$. Then $f=\sum_{n}f_ne_n$, where each $f_n\in\PC_a$. Defining $\mu_{m,n}=\mu_{e_m,e_n}$, a calculation shows that
	\[
	\int_{\TB}\left\langle\diff\mu(\zeta)\,  \frac{f(\zeta)}{1-r\zeta},\frac{f(\zeta)}{1-r\zeta}\right\rangle
	=
	\sum_{m,n}\int_{\TB} \frac{f_m(\zeta)\conj{f_n(\zeta)}}{|1-r\zeta|^2}\diff\mu_{m,n}(\zeta).
	\]
	For $|\mu_{m,n}|$-a.e. $\zeta\in\TB$, 
	\[
	\lim_{r\to 1^-}\frac{f_m(\zeta)\conj{f_n(\zeta)}}{|1-r\zeta|^2}
	=
	\frac{f_m(\zeta)\conj{f_n(\zeta)}}{|1-\zeta|^2}.
	\]
	Note that
	\[
	\frac{|f_m(\zeta)\conj{f_n(\zeta)}|}{|1-r\zeta|^2}
	\le 
	4\frac{|f_m(\zeta)\conj{f_n(\zeta)}|}{|1-\zeta|^2}.
	\]
	By Lemma \ref{lemma:VersionOfCauchySchwarz}, and Lemma \ref{lemma:WeakEmbedding}, the above right-hand side is $|\mu_{m,n}|$-integrable. Hence, by dominated convergence,
	\begin{multline*}
	\int_{\TB}\left\langle\diff\mu(\zeta)\,  \frac{f(\zeta)}{1-r\zeta},\frac{f(\zeta)}{1-r\zeta}\right\rangle
	=
	\sum_{m,n}\int_{\TB} \frac{f_m(\zeta)\conj{f_n(\zeta)}}{|1-r\zeta|^2}\diff\mu_{m,n}(\zeta)
	\\
	\to
	\sum_{m,n}\int_{\TB} \frac{f_m(\zeta)\conj{f_n(\zeta)}}{|1-\zeta|^2}\diff\mu_{m,n}(\zeta)
	=
	\int_{\TB}\left\langle \frac{\diff\mu(\zeta)}{|1-\zeta|^2}\, f(\zeta),f(\zeta)\right\rangle.
	\end{multline*}
	This proves that if $f\in\PC_a(\EC)$, and $\mu$ satisfies \eqref{eq:WeakEmbeddingCondition}, then \eqref{eq:Claim2} holds.
	
	The proof of \eqref{eq:Claim3} is similar. Reusing the above notation,
	\[
	\DC_{\mu}(k_rf)=\sum_{m,n}\DC_{\mu_{m,n}}(k_rf_m,k_rf_n).
	\]
	By Lemma \ref{lemma:FubiniForDirichletIntegrals}, each one of these Dirichlet integrals can be computed as
	\begin{align*}
	\DC_{\mu_{m,n}}(k_rf_m,k_rf_n)=\frac{1}{2\pi}\int_{\TB}\DC_\zeta(k_rf_m,k_rf_n)\diff\mu_{m,n}(\zeta).
	\end{align*}
	By polarization, and Lemma \ref{lemma:Convergence2}, 
	\begin{multline*}
	\DC_\zeta(k_rf_m,k_rf_n)=\frac{1}{4}\sum_{\sigma^4=1}\sigma\DC_\zeta\left(k_r(f_m+\sigma f_n)\right)
	\\
	\to 
	\frac{1}{4}\sum_{\sigma^4=1}\sigma\DC_\zeta\left(f_m+\sigma f_n\right)=\DC_\zeta(f_m,f_n),
	\end{multline*}
	for $|\mu_{m,n}|$-a.e $\zeta\in\TB$. Moreover, 
	\begin{multline*}
	|\DC_{\zeta}(k_rf_m,k_rf_n)|
	\le \DC_{\zeta}(k_rf_m)^{1/2}\DC_{\zeta}(k_rf_n)^{1/2}
	\\
	\le 
	\left(8\DC_{\zeta}(f_m)+4\frac{|f_m(\zeta)|^2}{|1-\zeta|^2}\right)^{1/2}
	\left(8\DC_{\zeta}(f_n)+4\frac{|f_n(\zeta)|^2}{|1-\zeta|^2}\right)^{1/2}.
	\end{multline*}
	By Lemma \ref{lemma:VersionOfCauchySchwarz}, and Lemma \ref{lemma:WeakEmbedding}, the right-hand side is $|\mu_{m,n}|$-integrable. By dominated convergence,
	\begin{multline*}
	\DC_{\mu}(k_rf)
	=
	\sum_{m,n}\frac{1}{2\pi}\int_{\TB}\DC_\zeta(k_rf_m,k_rf_n)\diff\mu_{m,n}(\zeta)
	\\
	\to
	\sum_{m,n}\frac{1}{2\pi}\int_{\TB}\DC_\zeta(f_m,f_n)\diff\mu_{m,n}(\zeta)
	=
	\DC_{\mu}(f).
	\end{multline*}
	This shows that \eqref{eq:Claim3} holds whenever $f\in\PC_a(\EC)$, and $\mu$ satisfies \eqref{eq:WeakEmbeddingCondition}. The proof that $(ii)\Rightarrow(iii)$ is complete.
\end{proof}	

\begin{proof}[$(iii)\Rightarrow(i)\,$\textnormal{:}]
	
	We are assuming the existence of $w_2\ge 0$, such that
	\begin{align}\label{eq:StrongEmbeddingCondition2}
	\frac{1}{2\pi}\int_{\TB}\left\langle \frac{\diff\mu(\zeta)}{|1-\zeta|^2}\, f(\zeta),f(\zeta)\right\rangle \le w_2\|f\|_{\DC_\mu^2(\EC)}^2,\quad f\in \PC_a(\EC).
	\end{align}
	Recall that $\phi_t:z\mapsto \exp\left(t(z+1)/(z-1)\right)$. The core of our proof is the following formula:
	\begin{lem}\label{lemma:MultiplicationFormula}
		Let $\mu$ be an $\LC_+$-valued measure that satisfies \eqref{eq:StrongEmbeddingCondition2}. If $f\in \PC_a(\EC)$, and $t> 0$, then
		\begin{align*}
		\|\phi_tf\|_{\DC_{\mu}^2(\EC)}^2=\|f\|_{\DC_{\mu}^2(\EC)}^2+\frac{t}{\pi}\int_{\TB}\left\langle \frac{ \diff\mu(\zeta)}{|1-\zeta|^2}\, f(\zeta),f(\zeta)\right\rangle.
		\end{align*}
	\end{lem}
	
	\begin{proof}
		Reusing some previous notation, $f=\sum f_ne_n$, where each $f_n\in\PC_a$, and $\{e_n\}$ is a finite orthonormal set. Moreover, $\mu_{m,n}=\mu_{e_m,e_n}$. From our discussion of the case $\EC=\CB$, in particular \eqref{eq:11111}, we know that
		\[
		\DC_{\mu_{n,n}}(\phi_tf_n)=\DC_{\mu_{n,n}}(f_n)+\frac{t}{\pi}\int_\TB \frac{|f_n(\zeta)|^2 }{|1-\zeta|^2}\diff\mu_{n,n}(\zeta).
		\]
		Applying \eqref{eq:StrongEmbeddingCondition2} to the function $f_ne_n$, the last integral is finite, so $\phi_tf_n\in\DC_{\mu_{n,n}}^2$, or equivalently, $\phi_t f_ne_n\in\DC_{\mu}^2(\EC)$.	By the Cauchy--Schwarz inequality, this implies that each integral $\DC_\mu(\phi_tf_me_m,\phi_tf_ne_n)$ is finite, as is
		\[
		\int_\TB \frac{f_m(\zeta)\conj{f_n(\zeta)} }{|1-\zeta|^2}\diff\mu_{m,n}(\zeta).
		\]
		Summing over $m$ and $n$ yields
		\[
		\DC_{\mu}(\phi_tf)=\DC_{\mu}(f)+\frac{t}{\pi}\int_{\TB}\left\langle \frac{ \diff\mu(\zeta)}{|1-\zeta|^2}\, f(\zeta),f(\zeta)\right\rangle.
		\]
		To reach the desired conclusion, we add $\|\phi_t f\|_{H^2(\EC)}^2=\|f\|_{H^2(\EC)}^2$ on both sides.	
	\end{proof}
	
	Let $t>0$. By Lemma \ref{lemma:MultiplicationFormula}, $M_{\phi_t}:\PC_a(\EC)\to \DC_{\mu}^2(\EC)$ is $\DC_{\mu}^2(\EC)$-bounded. As such, it has a unique bounded extension $\tilde M:\DC_{\mu}^2(\EC)\to \DC_{\mu}^2(\EC)$. By \eqref{eq:ReproducingProperty}, if $f_n\to f$ in $H^2(\EC)$, then $f_n(z)\to f(z)$ for each $z\in\DB$. Since convergence in $\DC_{\mu}^2(\EC)$ implies convergence in $H^2(\EC)$, we obtain that $\tilde M$ is indeed given by multiplication by $\phi_t$. Hence, $(M_{\phi_t})_{t\ge 0}\subset \LC(\DC_{\mu}^2(\EC))$. It is clear that $(M_{\phi_t})_{t\ge 0}$ is a semigroup. 
	
	We need to prove that $(M_{\phi_t})_{t\ge 0}$ is strongly continuous. If $f\in\PC_a(\EC)$, then Lemma~\ref{lemma:MultiplicationFormula} implies that 
	\begin{align}\label{eq:MultiplicationFormula}
	\|\phi_tf\|_{\DC_{\mu}^2(\EC)}^2=\|f\|_{\DC_{\mu}^2(\EC)}^2+t\left\langle\beta_{1}(M_{\phi_1})f,f\right\rangle_{\DC_{\mu}^2(\EC)},
	\end{align}
	and by continuity, this identity extends to $f\in\DC_{\mu}^2(\EC)$. We conclude that, for fixed $f\in\DC_{\mu}^2(\EC)$, the family $(\phi_t f)_{0<t<1}$ is bounded in $\DC_{\mu}^2(\EC)$. As we let $t\to 0^+$, a subsequence of $(\phi_t f)_{0<t<1}$ will converge weakly to some $g\in\DC_{\mu}^2(\EC)$. If $z\in\DB$, and $x\in\EC$, then \eqref{eq:ReproducingProperty} implies that $f\mapsto \left\langle f(z),x\right\rangle$ is a bounded linear functional on $\DC_{\mu}^2(\EC)$, so $\left\langle \phi_t f(z),x\right\rangle \to \left\langle g(z),x\right\rangle$ for said subsequence. But the point-wise limit of $\phi_tf$ is $f$, so $g=f$. As this uniquely determines the limit of any subsequence, the entire family converges weakly to $f$. Since \eqref{eq:MultiplicationFormula} also implies that $\lim_{t\to 0^+}\|\phi_tf\|_{\DC_\mu^2(\EC)}^2=\|f\|_{\DC_\mu^2(\EC)}^2$, we conclude that $\lim_{t\to 0^+}\phi_tf=f$ with convergence in $\DC_{\mu}^2(\EC)$. This establishes that $(M_{\phi_t})_{t\ge0}$ is a $C_0$-semigroup, i.e. $(iii)\Rightarrow(i)$.

	It remains to prove that $(M_{\phi_t})_{t\ge0}$ is quasicontractive, with parameter $w_2$. It follows from \eqref{eq:StrongEmbeddingCondition2} and Lemma~\ref{lemma:MultiplicationFormula} that 
	\[
	\|M_{\phi_t}\|^2\le 1+2w_2t,\quad t\ge 0.
	\]
	The right-hand side is the tangent at $t=0$ of the convex function $t\mapsto \exp(2w_2t)$. Hence, $(M_{\phi_t})_{t\ge 0}$ is quasicontractive with parameter $w_2$.
	
\end{proof}

\section{Examples}\label{sec:Examples}

The following example establishes the existence of $2$-isometric semigroups:

\begin{ex}
	Let $h:\TB\to\LC_+$ be a function such that $\zeta\mapsto \frac{\|h(\zeta)\|}{|1-\zeta|^{2}}$ is bounded, and $\zeta\mapsto \left\langle h(\zeta)x,y\right\rangle$ is $\lambda$-measurable for each $x,y\in\EC$. Defining $\mu\ge 0$ by $\left\langle \mu(E)x,y\right\rangle =\int_E \left\langle h(\zeta)x,y\right\rangle \diff\lambda(\zeta)$, we have that
	\begin{multline*}
	\frac{1}{2\pi}\int_{\TB}\left\langle \frac{\diff\mu(\zeta)}{|1-\zeta|^2}\, f(\zeta),f(\zeta)\right\rangle
	=
	\frac{1}{2\pi}\int_{\TB} \frac{\|h(\zeta)f(\zeta)\|^2}{|1-\zeta|^2}\diff\lambda (\zeta)
	\\
	\lesssim
	\|f\|_{H^2(\EC)}^2\le \|f\|_{\DC_\mu^2(\EC)}^2,\quad f\in\PC_a(\EC).
	\end{multline*}
	Therefore, condition $(iii)$ of Theorem~\ref{thm:MainMeasures} is satisfied, and $(M_{\phi_t})_{t\ge 0}\subset \LC(\DC_\mu^2(\EC))$ is a $2$-isometric semigroup.
\end{ex}

Not every $2$-isometry is the cogenerator of a $C_0$-semigroup:

\begin{ex}
	Let $\EC=\CB$, and $\mu=\lambda$. Then condition $(iii)$ of Theorem~\ref{thm:MainMeasures} is not satisfied, so $M_z:\DC_\lambda^2\to\DC_\lambda^2$ is not the cogenerator of a $C_0$-semigroup. Identifying an analytic function with its sequence of Maclaurin coefficients, $\DC_\lambda^2$ is isometrically isomorphic to the space of sequences $(a_k)_{k\ge 0}$ such that 
	\[
	\sum_{k=0}^{\infty}(1+k)|a_k|^2<\infty,
	\] 
	e.g. \cite[Chapter~1]{ElFallah-Kellay-Mashreghi-Ransford2014:APrimerOnTheDirichletSpace}. Under this identification, $M_z:\DC_\lambda^2\to\DC_\lambda^2$ is unitarily equivalent to the right-shift
	\[
	(a_0,a_1,a_2,\ldots)\mapsto (0,a_0,a_1,\ldots ).
	\]
	Hence, the operator $M_z:\DC_\lambda^2\to\DC_\lambda^2$ is arguably the simplest example of a non-isometric $2$-isometry.
\end{ex}

The above example shows that $M_z:\DC_\lambda^2\to\DC_\lambda^2$ does not satisfy \eqref{eq:QCCondCoGen} for any $w\ge 0$. Hence, the corresponding condition in Theorem~\ref{thm:MainConverse} is not superfluous. Similarly, condition $(ii)$ of Theorem~\ref{thm:MainGen} is also not superfluous. However, in order to see this we need a slightly more refined example:

Recall that $M_z-I:\DC_\mu^2(\EC)\to\DC_\mu^2(\EC)$ is injective for any operator measure $\mu\ge 0$, see Section~\ref{SubSec:AnalyticOperators}. We may therefore define the operator
\[
A=(M_z+I)(M_z-I)^{-1},\quad D(A)=(M_z-I)\DC_{\mu}^2(\EC).
\]
Since $A-I=2(M_z-I)^{-1}$ has the bounded inverse $\frac{1}{2}(M_z-I)$, we conclude that $A$ is closed, and $1\in\rho (A)$. Moreover, $M_z$ is $2$-isometric, so Lemma~\ref{lemma:M-IsometryToM-SkewSymmetry} yields that $A$ is $2$-skew-symmetric. If $\lambda>0$, then
\[
(\lambda-A)^{-1}=\frac{1}{1+\lambda}(M_z-I)\left(\frac{\lambda-1}{\lambda +1}M_z-I\right)^{-1}.
\]
The last factor exists as a bounded operator, because $\frac{\lambda -1}{\lambda +1}<1$, and $\sigma(M_z)=\clos{\DB}$. The operator $A$ is densely defined if and only if $M_z-I$ has dense range. Moreover, $A$ is the generator of a $C_0$-semigroup if and only if $M_z$ is the cogenerator.

\begin{ex}
	Define $\mu$ by $\diff\mu(\zeta)=|1-\zeta|\diff\lambda(\zeta)$. This measure violates condition $(iii)$ of Theorem~\ref{thm:MainMeasures}. Hence, $M_z:\DC_{\mu}^2\to\DC_{\mu}^2$ is not the cogenerator of a $C_0$-semigroup.
	
	We now prove that $M_z-I$ has dense range. For $z\in\clos{\DB}$, and $r\in (0,1)$, let $k_r(z)=\frac{1-z}{1-rz}$. If $f\in\DC_{\mu}^2$, then $k_rf\in (M_z-I)\DC_{\mu}^2$. We prove that $k_rf\to f$ in $\DC_{\mu}^2$ as $r\to 1^-$. It is clear that $\|f-k_rf\|_{H^2}\to 0$. Hence, we need to show that $\DC_{\mu}(f-k_rf)\to 0$.
	
	Since $f\in\DC_{\mu}^2$, $\DC_{\zeta}(f)<\infty$ for $\mu$-a.e. $\zeta\in\TB$. For such $\zeta$, let $F(z)=\frac{f(z)-f(\zeta)}{z-\zeta}$. From Lemma~\ref{lemma:Convergence2},
	\begin{align*}
	\DC_{\zeta}(f-k_rf)
	&\le 
	2\left(\|F-k_rF\|_{H^2}^2+\frac{1-r}{1+r}\frac{|f(\zeta)|^2}{|1-r\zeta|^2}\right).
	\end{align*}
	Except possibly for $\zeta=1$, $\DC_{\zeta}(f-k_rf)\to 0$ as $r\to 1^-$. Moreover, $\|F-k_rF\|_{H^2}^2$ is bounded, while 
	\[
	\frac{1-r}{1+r}\frac{|f(\zeta)|^2}{|1-r\zeta|^2}
	\le
	4\frac{|f(\zeta)|^2}{|1-\zeta|}.
	\]
	This right-hand side is $\mu$-integrable, because $\DC_{\mu}^2\subset H^2$. By Fubini's theorem \eqref{eq:FubiniForPositiveMeasures}, dominated convergence implies that
	\[
	\DC_{\mu}(f-k_rf)
	=
	\frac{1}{\pi}\int_\TB \DC_\zeta (f-k_rf)\diff\mu(\zeta)\to 0.
	\]
	Therefore, $M_z-I$ has dense range.
	
	Based on our discussion, we conclude that the operator $A$, defined by
	\[
	A=(M_z+I)(M_z-I)^{-1},\quad D(A)=(M_z-I)\DC_{\mu}^2,
	\]
	is closed, densely defined, and $2$-skew-symmetric. Moreover, $\lambda -A:D(A)\to\HC$ is surjective for any $\lambda>0$. Nevertheless, $A$ fails to generate a $C_0$-semigroup, apparently for the sole reason that \eqref{eq:QCCondGen} is not satisfied for any $w\ge 0$.
\end{ex}

In the proof of Theorem~\ref{thm:MainMeasures}, recall our discussion of the case $\EC=\CB$ on p.~\pageref{p:ScalarCase}. The main part of the argument was to identify the measures $\mu$ for which the multiplication operators $(M_{\phi_t})_{t\ge 0}$ are bounded on $\DC_{\mu}^2$. With this in mind, we give an example of a semigroup that would appear to be $2$-isometric, apart from the fact that its elements are not bounded operators. This relates to some other recent examples of ``unbounded $2$-isometries'', \cite[Example~3.4]{Bermudez-Martinon-Muller2014:mq-IsometriesOnMetricSpaces}, and \cite[Example~7.1]{Rydhe2018:CyclicMIsometriesAndDirichletTypeSpaces}.

\begin{ex}
	Consider the right-shift semigroup $(S(t))_{t\ge 0}$, defined for functions $f:(0,\infty)\to\CB$ by
	\begin{align*}
	(S(t)f)(s)
	=
	\left\{
	\begin{array}{rr}
	f(s-t),& s\ge t,\\
	0, & t>s.
	\end{array}
	\right.
	\end{align*}
	Clearly, $S(0)f = f$, and $S(t_1)S(t_2)=S(t_1+t_2)$. Consider now the Hilbert space $L^2(0,\infty;s\diff s)$ of functions $f:(0,\infty)\to\CB$ such that
	\[
	\|f\|_1^2:=\int_{s=0}^\infty |f(s)|^2s\diff s<\infty.
	\]
	If $f$ is the indicator function of the interval $(0,h)$, then $\|f\|_1^2=\frac{h^2}{2}$, whereas $\|S(t)f\|_1^2=th+\frac{h^2}{2}$. Hence, any non-trivial right-shift $S(t)$ fails to be bounded on $L^2(0,\infty;s\diff s)$. On the other hand, if $f\in L^2(0,\infty;s\diff s)\cap L^2_{loc}(0,\infty)$, which is the natural domain of $S(t)$, then $\|S(t)^2f\|_1^2-2\|S(t)f\|_1^2+\|f\|_1^2=0$.
\end{ex}

\section{$m$-concave semigroups}\label{sec:m-ConcaveSemigroups}

Let $(T_t)_{t\ge 0}$ be a $C_0$-semigroup. We say that $(T_t)_{t\ge 0}$ is \textit{$m$-concave} if each $T_t$ is $m$-concave, i.e. $\beta_m(T_t)\le 0$. One can use \eqref{eq:HigherOrderDefectAsIntegral} to obtain:
\begin{prop}\label{prop:m-SubSkewSymmetricGenerator}
	Let $m\in\ZB_{\ge 1}$, and $(T_t)_{t\ge 0}$ be a $C_0$-semigroup with generator $A$. Then $(T_t)_{t\ge 0}$ is $m$-concave if and only if 
	\[
	\alpha_m^A(y)\le 0,\quad y\in D(A^m),
	\]
	where $\alpha_m^A$ is given by \eqref{eq:DefinitionOfAlphaM}.
\end{prop}

For $m$-concave operators, the following analogue of \eqref{eq:PolynomialGrowthDiscreteTime} holds:

\begin{prop}\label{prop:PolynomialDominationDiscreteTime}
	Let $m\in\ZB_{\ge 1}$, and assume that $T\in\LC$ is $m$-concave. Then
	\begin{align}\label{eq:PolynomialDominationDiscreteTime}
	\|T^kx\|^2\le \sum_{j=0}^{m-1}\binom{k}{j}\left\langle \beta_{j}(T)x,x\right\rangle ,\quad x\in\HC.
	\end{align}
\end{prop}
\begin{proof}
	For $k<m$ we have equality, by \eqref{eq:SumOfDefectOperators}. We therefore consider $k\ge m$:
	\[
	T^{*k}T^k=\sum_{j=0}^{m-1}\binom{k}{j}\beta_{j}(T)+\sum_{j=m}^{k}\binom{k}{j}\beta_{j}(T).
	\]
	It suffices to show that the second sum is $\le 0$.	From \cite[Proposition~2.1]{Rydhe2018:CyclicMIsometriesAndDirichletTypeSpaces}, we take the formula
	\[
	\beta_{j+m}(T)=\sum_{i=0}^j(-1)^{j-i}\binom{j}{i}T^{*i}\beta_m(T)T^i.
	\]
	Shifting the index, and changing the order of summation, the previous formula yields
	\begin{align*}
	\sum_{j=m}^{k}\binom{k}{j}\beta_{j}(T)
	&=
	\sum_{j=0}^{k-m}\binom{k}{j+m}\beta_{j+m}(T)
	\\
	&=
	\sum_{j=0}^{k-m}\binom{k}{j+m}\sum_{i=0}^j(-1)^{j-i}\binom{j}{i}T^{*i}\beta_m(T)T^i
	\\
	&=
	\sum_{i=0}^{k-m}\left(\sum_{j=i}^{k-m}(-1)^{j-i}\binom{k}{j+m}\binom{j}{i}\right)T^{*i}\beta_m(T)T^i.
	\end{align*}
	Since $\beta_m(T)\le 0$ by assumption, we are done if $\sum_{j=i}^{N}(-1)^{j-i}\binom{N+m}{j+m}\binom{j}{i}\ge 0$ for all integers $i$ and $N$ with $N\ge i\ge 0$. This follows by Lemma~\ref{lemma:BinomialSum} below.
\end{proof}
\begin{lem}\label{lemma:BinomialSum}
	Given integers $N\ge i\ge 0$, $m\ge 1$, it holds that
	\[
	\sum_{j=i}^{N}(-1)^{j-i}\binom{N+m}{j+m}\binom{j}{i}=\binom{m-1+N-i}{N-i}.
	\]
\end{lem}
\begin{proof}
	Define 
	\begin{align*}
	\LHS_{i,N}^{(m)}=\sum_{j=i}^{N}(-1)^{j-i}\binom{N+m}{j+m}\binom{j}{i},
	\end{align*}
	and
	\begin{align*}
	\RHS_{i,N}^{(m)}=\binom{m-1+N-i}{N-i}.
	\end{align*}	
	Step 1, ($i=0$)\,: Using \eqref{eq:BinomialIdentity},
	\begin{align*}
	\LHS_{0,N}^{(m)}
	&=
	\sum_{j=0}^{N}(-1)^{j}\left[\binom{N+m-1}{j+m-1}+\binom{N+m-1}{j+m}\right]
	\\
	&=
	\sum_{j=-1}^{N-1}(-1)^{j+1}\binom{N+m-1}{j+m}+\sum_{j=0}^{N}(-1)^{j}\binom{N+m-1}{j+m}
	\\
	&=
	\binom{N+m-1}{m-1}+(-1)^N\binom{N+m-1}{N+m}
	\\
	&=
	\binom{N+m-1}{N}=\RHS_{0,N}^{(m)}.
	\end{align*}	
	Step 2, (recursion formula)\,: Using \eqref{eq:BinomialIdentity} again,
	\begin{align*}
	\LHS_{i,N}^{(m)}
	={}&
	\sum_{j=i}^{N}(-1)^{j-i}\binom{N+m}{j+m}\left[\binom{j-1}{i-1}+\binom{j-1}{i}\right]
	\\
	={}&
	\sum_{j=i-1}^{N-1}(-1)^{j+1-i}\binom{N+m}{j+m+1}\binom{j}{i-1}
	\\
	&+\sum_{j=i}^{N-1}(-1)^{j+1-i}\binom{m+N}{j+m+1}\binom{j}{i}
	\\
	={}&
	\LHS_{i-1,N-1}^{(m+1)}-\LHS_{i,N-1}^{(m+1)}.
	\end{align*}	
	Step 3, (induction)\,: By step 1, it holds that $\LHS_{0,N}^{(m)}=\RHS_{0,N}^{(m)}$ whenever $N\ge 0$, and $m\ge 1$. Moreover, $\LHS_{N,N}^{(m)}=\RHS_{N,N}^{(m)}$ for any $N\ge 0$, $m\ge 1$. We may thus assume that there exists $i\ge 0$ such that
	\[
	m\ge 1,N\ge i\quad \Rightarrow \quad \LHS_{i,N}^{(m)}=\RHS_{i,N}^{(m)},
	\]
	and $N\ge i+1$ such that
	\[
	m\ge 1\quad \Rightarrow \quad \LHS_{i+1,N}^{(m)}=\RHS_{i+1,N}^{(m)}.
	\]
	By step 2, this implies that
	\begin{align*}
	\LHS_{i+1,N+1}^{(m)}
	&=
	\LHS_{i,N}^{(m+1)}-\LHS_{i+1,N}^{(m+1)}
	\\
	&=
	\RHS_{i,N}^{(m+1)}-\RHS_{i+1,N}^{(m+1)} 
	=
	\RHS_{i+1,N+1}^{(m)}.
	\end{align*}
	By induction over $N\ge i+1$,
	\[
	m\ge 1,N\ge i+1\quad \Rightarrow \quad \LHS_{i+1,N}^{(m)}=\RHS_{i+1,N}^{(m)}.
	\]
	By induction over $i\ge 0$, 
	\[
	m\ge 1,N\ge i\ge 0\quad \Rightarrow \quad \LHS_{i,N}^{(m)}=\RHS_{i,N}^{(m)}.
	\]
\end{proof}

Proposition~\ref{prop:PolynomialDominationDiscreteTime} implies an analogue of Theorem~\ref{thm:PolynomialGrowthContinuousTime}:

\begin{prop}\label{prop:PolynomialDominationContinuousTime}
	Let $(T_t)_{t\ge 0}$ be an $m$-concave semigroup with generator $A$. If $y\in D(A^{m-1})$, then
	\begin{align*}
	\|T_ty\|^2\le\sum_{j=0}^{m-1}\frac{t^j}{j!}\alpha_j^A(y),\quad t\ge 0,
	\end{align*}
	where $(\alpha_{j}^A)_{j=0}^{m-1}$ is given by \eqref{eq:DefinitionOfAlphaM}.
\end{prop}

A significant difference from the $m$-isometric case is that we have no reason to expect the forms $(\alpha_j^A)_{j=0}^{m-1}$ to be bounded. Therefore, we obtain no evidence that $m$-concave semigroups are quasicontractive by necessity. On the other hand, from Proposition~\ref{prop:PolynomialDominationDiscreteTime}, we have that
\[
T_k^*T_k\le \sum_{j=0}^{m-1}\binom{k}{j}\beta_{j}(T_1).
\]
Each form $x\mapsto \left\langle \beta_j(T_1)x,x\right\rangle$ is bounded on $\HC$. Together with the semigroup property, this implies that $\|T_t\|^2\lesssim (1+t)^{m-1}$. From \eqref{eq:LaplaceTransformOfSemigroup}, we therefore obtain:
\begin{prop}
	Let $m\in\ZB_{\ge 0}$, and $(T_t)_{t\ge 0}$ be an $m$-concave semigroup with generator $A$. Then
	\[
	\sigma(A)\subseteq \left\{ z\in\CB;\,\Re z\le 0 \right\}.
	\]
	In particular, $(T_t)_{t\ge 0}$ has a well-defined cogenerator
	\[
	T=(A+I)(A-I)^{-1}\in\LC.
	\]
\end{prop}

Combining the above result with Proposition~\ref{prop:m-SubSkewSymmetricGenerator}, and Lemma~\ref{lemma:M-IsometryToM-SkewSymmetry}, one obtains:

\begin{prop}
	A $C_0$-semigroup $(T_t)_{t\ge0}$ is $m$-concave if and only if it possess an $m$-concave cogenerator.
\end{prop}

\section*{Acknowledgements}
During the preparation of this work, I have enjoyed interesting conversations with Jonathan Partington, Bartosz Malman, and Mikael Persson Sundqvist. I am also indebted to the anonymous referee for carefully reading and commenting on several versions of this manuscript. In particular, the suggestion to include a more thorough discussion on generators led to substantial improvements not only of the presentation, but also of the results.


\providecommand{\bysame}{\leavevmode\hbox to3em{\hrulefill}\thinspace}
\providecommand{\MR}{\relax\ifhmode\unskip\space\fi MR }
\providecommand{\MRhref}[2]{%
	\href{http://www.ams.org/mathscinet-getitem?mr=#1}{#2}
}
\providecommand{\href}[2]{#2}

\end{document}